%% file: main_part3.tex
\pdfoutput=1
\documentclass[onefignum,onetabnum,reqno]{siamart220329}

\usepackage{stmaryrd}
\usepackage{float}
\usepackage{microtype}
\usepackage{upgreek}
\usepackage{mathtools}
\usepackage[rose]{paper_diening}
\usepackage{mathrsfs}
\usepackage{enumitem}
\usepackage{hyperref}
\usepackage{booktabs}
\usepackage{colortbl}
\usepackage{tabulary}
\usepackage{diagbox}
\usepackage{caption}

\newsiamremark{assumption}{Assumption}
\input{part3_shared}

\usepackage{xcolor}
\definecolor{rltred}{rgb}{0.75,0,0}
\definecolor{rltgreen}{rgb}{0,0.5,0}
\definecolor{rltblue}{rgb}{0,0,0.75}

\ifpdf
\hypersetup{
  pdftitle={A Local Discontinuous Galerkin Approximation for the $p$-Navier--Stokes system},
  pdfauthor={A. Kaltenbach, M. \Ruzicka}
}
\fi

\providecommand{\meantmp}[2]{#1\langle{#2}#1\rangle}
\providecommand{\mean}[1]{\meantmp{}{#1}}

\providecommand{\jumptmp}[2]{#1\llbracket{#2}#1\rrbracket}
\providecommand{\jump}[1]{\jumptmp{}{#1}}

\providecommand{\avgtmp}[2]{#1\{{#2}#1\}}
\providecommand{\avg}[1]{\avgtmp{}{#1}}
\providecommand{\bigavg}[1]{\avgtmp{\big}{#1}}

\providecommand{\flux}[1]{{\widehat{#1}}}

\providecommand{\PiDG}{{\Uppi_{h}^{k}}}

\providecommand{\Pia}{{\Uppi_h^{0}}}

\providecommand{\Vo}{\smash{\mathaccent23 V}}
\providecommand{\Qo}{\smash{\mathaccent23 Q}}
\providecommand{\Xhk}{\smash{X_h^k}}
\providecommand{\Vhk}{\smash{V_h^k}}

\providecommand{\Qhkc}{\smash{Q_{h,c}^k}}
\providecommand{\Qhko}{\smash{{\mathaccent23 Q}_h^k}}
\providecommand{\Qhkco}{\smash{{\mathaccent23 Q}_{h,c}^k}}

\providecommand{\SSS}{\boldsymbol{\mathcal{S}}}

\newcommand{\Ghk}{\boldsymbol{\mathcal{G}}_h^k}

\newcommand{\Dhk}{\boldsymbol{\mathcal{D}}_h^k}

\newcommand{\Divhk}{\mathcal{D}\dot{\iota}\nu_h^k}

\newcommand{\Rhk}{\boldsymbol{\mathcal{R}}_h^k}

\newcommand{\WDG}{W^{1,p}(\mathcal{T}_h)}

\providecommand{\divo}{\textrm{div}\,}

\providecommand{\sss}{\avg{\abs{\Pia \Dhk \bfv_h}}}
\providecommand{\sssl}{\avg{\abs{\Pia \Dhk \bfv_h}}}

\begin{document}

\maketitle

\begin{abstract}
  In the present paper, we prove convergence rates for the pressure of
  the Local Discontinuous Galerkin (LDG) approximation, proposed in
  Part I of the paper (cf.~\cite{kr-pnse-ldg-1}), of systems of
  $p$-Navier--Stokes~type and $p$-Stokes type with $p\in
  (2,\infty)$. 
  The results are supported by numerical experiments.
\end{abstract}



\section{Introduction}

In this paper, we continue our study of the Local Discontinuous Galerkin
(LDG) scheme, proposed in Part I of the paper (cf.~\cite{kr-pnse-ldg-1}), of
steady systems~of $p$-Navier--Stokes type. In this paper, as we already did
in Part II  of the paper (cf.~\cite{kr-pnse-ldg-2}), we restrict
ourselves to the homogeneous problem, i.e., 
\begin{equation}
  \label{eq:p-navier-stokes}
  \begin{aligned}
    -\divo\SSS(\bfD\bfv)+[\nabla\bfv]\bfv+\nabla q&=\bfg  \qquad&&\text{in }\Omega\,,\\
    \divo\bfv&=0 \qquad&&\text{in }\Omega\,,
    \\
    \bfv &= \mathbf{0} &&\text{on } \partial\Omega\,.
  \end{aligned}
\end{equation}
This system describes the steady motion of a homogeneous,
incompressible~fluid~with shear-dependent viscosity. More precisely,
for a given vector field $\bfg:\Omega\to \setR^d$ describing external
body forces and a homogeneous Dirichlet
boundary~condition~\eqref{eq:p-navier-stokes}$_3$, we seek for a
velocity vector field~${\bfv=(v_1,\dots,v_d)^\top\colon \Omega\to
  \setR^d}$~and~a~scalar~kinematic~pressure ${q\colon \Omega\hspace{-0.15em}\to\hspace{-0.15em} \setR}$ solving~\eqref{eq:p-navier-stokes}.
Here, $\Omega\hspace{-0.15em}\subseteq\hspace{-0.15em} \mathbb{R}^d$, $d\hspace{-0.15em}\in\hspace{-0.15em} \set{2,3}$, is a bounded~polyhedral~domain~\mbox{having} a Lipschitz continuous boundary $\partial\Omega$. The extra stress tensor $\SSS(\bfD\bfv)\colon \Omega\to \setR^{d\times d}_{\textup{sym}}$ depends on the strain rate tensor $\smash{\bfD\bfv\coloneqq \frac{1}{2}(\nabla\bfv+\nabla\bfv^\top)\colon \Omega\to \setR^{d\times d}_{\textup{sym}}}$, i.e., the symmetric part of the velocity tensor $\bfL\coloneqq \nabla\bfv
\colon \Omega\to \setR^{d\times d}$.~The~convective~term~$\smash{[\nabla\bfv]\bfv\colon \Omega\to \mathbb{R}^d}$ is defined via $\smash{([\nabla\bfv]\bfv)_i\coloneqq \sum_{j=1}^d{v_j\partial_j v_i}}$ for all $i=1,\dots,d$.

Throughout the paper, we  assume that the extra stress tensor~$\SSS$~has~\mbox{$(p,\delta)$-structure} (cf.~Assumption~\ref{assum:extra_stress}). The relevant example falling into this class is 
\begin{align*}
    \SSS(\bfD\bfv)=\mu\, (\delta+\vert \bfD\bfv\vert)^{p-2}\bfD\bfv\,,
\end{align*}
where $p\in (1,\infty)$, $\delta\ge 0$, and $\mu>0$.

For a discussion of the model and the state of the art, we refer to
Part~I~of~the~paper (cf.~\cite{kr-pnse-ldg-1}). As already pointed
out, to the best of the authors' knowledge, there are no
investigations using DG methods for the $p$-Navier--Stokes problem
\eqref{eq:p-navier-stokes}.~In~this~paper, we continue the
investigations of Part I and Part II of the paper
(cf.~\cite{kr-pnse-ldg-1, kr-pnse-ldg-2}), and
prove~\mbox{convergence}~rates for the pressure of the homogeneous
$p$-Navier--Stokes problem \eqref{eq:p-navier-stokes} under the
assumption~that~the velocity and $\bfg$ satisfy natural regularity
conditions and a 
smallness~condition for the velocity~in~the~energy~norm. In doing so,
we restrict ourselves~to~the~case~${p \in (2,\infty)}$. Our approach
is inspired by the results in \cite{dkrt-ldg}, \cite{kr-phi-ldg}, and
\cite{bdr-phi-stokes}. The same results are obtained for the
$p$-Stokes problem without the~smallness~condition.  We would
  like to point out that there are no results in the literature
  proving convergence rates for the pressure for the
  $p$-Navier--Stokes equations \eqref{eq:p-navier-stokes} ($p\neq 2$)
  neither for DG methods nor for FE methods. Even in the case of the
  $p$-Stokes problem ($p\neq 2$) there is only one result for DG
  methods and some for FE methods (cf.~Remark \ref{rem:pse}).

\textit{This paper is organized as follows:} \!In Section
\ref{sec:preliminaries}, we introduce the employed~\mbox{notation}, define
 relevant function spaces, 
 basic assumptions on the extra stress~tensor~$\SSS$~and~its consequences, the weak formulations Problem (Q) and
Problem~(P)~of~the~system~\eqref{eq:p-navier-stokes}, and the  discrete
operators.  In Section \ref{sec:ldg}, we define our numerical
fluxes~and~derive~the flux \hspace{-0.1mm}and \hspace{-0.1mm}the \hspace{-0.1mm}primal
\hspace{-0.1mm}formulation, \hspace{-0.1mm}i.e, \hspace{-0.1mm}Problem~\hspace{-0.1mm}(Q$_h$) \hspace{-0.1mm}and \hspace{-0.1mm}Problem~\hspace{-0.1mm}(P$_h$),~\hspace{-0.1mm}of~\hspace{-0.1mm}the~\hspace{-0.1mm}\mbox{system}~\hspace{-0.1mm}\eqref{eq:p-navier-stokes}.  In Section
\ref{sec:rates}, we~\mbox{derive} error estimates for our problem
(cf.~Theorem~\ref{thm:error}, Corollary~\ref{cor:error}).  These
are the first convergence rates for a DG-method for systems of
$p$-Navier--Stokes type. In Section \ref{sec:experiments}, we present 
numerical experiments.

\section{Preliminaries}\label{sec:preliminaries}

\subsection{Function \hspace*{-0.1mm}spaces}

\!We \hspace*{-0.1mm}use \hspace*{-0.1mm}the \hspace*{-0.1mm}same \hspace*{-0.1mm}notation \hspace*{-0.1mm}as \hspace*{-0.1mm}in \hspace*{-0.1mm}Part~\hspace*{-0.1mm}I~\hspace*{-0.1mm}of~\hspace*{-0.1mm}the~\hspace*{-0.1mm}paper~\hspace*{-0.1mm}(cf.~\hspace*{-0.1mm}\cite{kr-pnse-ldg-1}). For the convenience of the reader, we repeat some of it.

We employ $c, C>0$ to denote generic constants, that may change from line
to line, but are not depending on the crucial quantities. For $k\in \setN$ and $p\in [1,\infty]$, we employ the customary
Lebesgue spaces $(L^p(\Omega), \smash{\norm{\cdot}_p}) $ and Sobolev
spaces $(W^{k,p}(\Omega), \smash{\norm{\cdot}_{k,p}})$, where $\Omega \hspace{-0.1em}\subseteq \hspace{-0.1em} \setR^d$, $d \hspace{-0.1em}\in \hspace{-0.1em} \set{2,3}$, is a bounded,
polyhedral Lipschitz domain.~The~space~$\smash{W^{1,p}_0(\Omega)}$
is defined as the space of functions from $W^{1,p}(\Omega)$ whose trace vanishes on $\partial\Omega$.~We~equip $\smash{W^{1,p}_0(\Omega)}$ 
with the norm $\smash{\norm{\nabla\,\cdot\,}_p}$. 

We do not distinguish between spaces for scalar,
vector-~or~\mbox{tensor-valued}~functions. However, we always denote
vector-valued functions by boldface letters~and~tensor-valued
functions by capital boldface letters. The mean value of a locally
integrable function $f$ over a measurable set $M\subseteq \Omega$ is
denoted by
${\mean{f}_M\coloneqq \smash{\dashint_M f
    \,\textup{d}x}\coloneqq \smash{\frac 1 {|M|}\int_M f
    \,\textup{d}x}}$.~Moreover, we employ the notation
$\hskp{f}{g}\coloneqq \int_\Omega f g\,\textup{d}x$, whenever the
right-hand side is well-defined.

From the theory of Orlicz spaces (cf.~\cite{ren-rao}) and generalized
Orlicz spaces~(cf.~\cite{HH19}), we use N-functions
$\psi\colon  \setR^{\geq 0} \to \setR^{\geq 0}$ and generalized N-functions
$\psi \colon \Omega \times \setR^{\ge 0} \to \setR^{\ge 0}$,~i.e.,
$\psi$ is a Carath\'eodory function such that $\psi(x,\cdot)$ is an
N-function~for~a.e.~${x \in \Omega}$,~respectively. For  $f\in L^0(\Omega)$\footnote{Here, $L^0(\Omega)$ denotes the set of Lebesgue measurable scalar function defined on $\Omega$.\vspace{-2mm}}, the modular~is~defined~via
$\rho_\psi(f)\coloneqq \rho_{\psi,\Omega}(f)\coloneqq \int_\Omega
\psi(\abs{f})\,\textup{d}x $ if $\psi$ is an N-function and
$\rho_\psi(f)\coloneqq \rho_{\psi,\Omega}(f)\coloneqq \int_\Omega
\psi(\cdot,\abs{f})\,\textup{d}x $, if $\psi$
is~a~generalized~\mbox{N-function}. Then, for a (generalized)
N-function $\psi$, we~denote~by~${L^{\psi}(\Omega)\coloneqq \{f\hspace{-0.15em}\in\hspace{-0.15em} L^0(\Omega)\mid
\rho_\psi(f)\hspace{-0.15em}<\hspace{-0.15em}\infty\}}$, the (generalized) Orlicz space. Equipped with
the induced Luxembourg~norm, i.e.,
$\smash{\norm {f}_{\psi}}\coloneqq  \smash{\inf \set{\lambda >0\mid
    \rho_\psi(f/\lambda) \le 1}}$, the space (generalized) Orlicz space
$L^\psi(\Omega)$~is~a~Banach space.  If $\psi$ is a generalized
N-function, then, for every $f\in L^{\psi}(\Omega)$ and
$g\in L^{\psi^*}(\Omega)$, there holds the generalized~Hölder~inequality
\begin{align}\label{eq:gen_hoelder}
    (f,g)\leq 2\,\|f\|_{\psi}\|g\|_{\psi^*}\,.
\end{align}
An N-function
$\psi$~satisfies~the~\mbox{$\Delta_2$-condition}, 
if there exists
$K> 2$ such that for all
$t \ge 
0$, it holds ${\psi(2\,t) \leq K\,
  \psi(t)}$. We denote the smallest such constant by
$\Delta_2(\psi)\hspace*{-0.1em}>\hspace*{-0.1em}0$.
We need the following version of~the~\mbox{$\varepsilon$-Young} inequality: for every
${\varepsilon\!>\! 0}$,~there exits a constant $c_\epsilon\!>\!0 $, depending~only~on~$\Delta_2(\psi),\Delta_2( \psi ^*)\!<\!\infty$, such that~for~every~${s,t\!\geq\! 0}$,~it~holds
\begin{align}
  \label{ineq:young}
  \begin{split}
    t\,s&\leq \epsilon \, \psi(t)+ c_\epsilon \,\psi^*(s)\,.
  \end{split}
\end{align}

\subsection{Basic \hspace*{-0.1mm}properties \hspace*{-0.1mm}of \hspace*{-0.1mm}the \hspace*{-0.1mm}extra \hspace*{-0.1mm}stress \hspace*{-0.1mm}tensor}

\!\!Throughout~\hspace*{-0.1mm}the~\hspace*{-0.1mm}entire~\hspace*{-0.1mm}\mbox{paper}, we always assume that the extra stress tensor 
$\SSS$
has $(p,\delta)$-structure, which is defined here in a more stringent
way compared to Part I of the paper (cf.~\cite{kr-pnse-ldg-1}). A detailed
discussion and full proofs can be found, e.g., in
\cite{die-ett,dr-nafsa}. For a given tensor $\bfA\in \setR^{d\times d}$, we denote its symmetric part by
${\bfA^{\textup{sym}}\coloneqq \frac{1}{2}(\bfA+\bfA^\top)\in
  \setR^{d\times d}_{\textup{sym}}\coloneqq \{\bfA\in \setR^{d\times
    d}\mid \bfA=\bfA^\top\}}$.

For $p \in (1,\infty)$~and~$\delta\ge 0$, we define a special N-function
$\phi\coloneqq \phi_{p,\delta}\colon\setR^{\ge 0}\to \setR^{\ge 0}$~by
\begin{align} 
  \label{eq:def_phi} 
  \varphi(t)\coloneqq  \int _0^t \varphi'(s)\, \mathrm ds,\quad\text{where}\quad
  \varphi'(t) \coloneqq  (\delta +t)^{p-2} t\,,\quad\textup{ for all }t\ge 0\,.
\end{align}
The properties of $\phi$ are discussed in detail in \cite{die-ett,dr-nafsa,kr-pnse-ldg-1}.

An important tool in our analysis play {\rm shifted N-functions}
$\{\psi_a\}_{\smash{a \ge 0}}$,~cf.~\cite{DK08,dr-nafsa}. For a given N-function $\psi\colon\mathbb{R}^{\ge 0}\to \mathbb{R}^{\ge
  0}$, we define the family  of shifted N-functions ${\psi_a\colon\mathbb{R}^{\ge
    0}\to \mathbb{R}^{\ge 0}}$,~${a \ge 0}$,  via
\begin{align}
  \label{eq:phi_shifted}
  \psi_a(t)\coloneqq  \int _0^t \psi_a'(s)\, \mathrm ds\,,\quad\text{where }\quad
  \psi'_a(t)\coloneqq \psi'(a+t)\frac {t}{a+t}\,,\quad\textup{ for all }t\ge 0\,.
\end{align}

\begin{assumption}[Extra stress tensor]\label{assum:extra_stress} We assume that the extra stress tensor $\SSS\colon\setR^{d\times d}\to \setR^{d\times d}_{\textup{sym}}$ belongs to $C^0(\setR^{d\times d}; \setR^{d\times d}_{\textup{sym}})\cap C^1(\setR^{d\times d}\setminus\{\mathbf{0}\}; \setR^{d\times d}_{\textup{sym}}) $ and satisfies $\SSS(\bfA)=\SSS(\bfA^{\textup{sym}})$ for all $\bfA\in \setR^{d\times d}$ and $\SSS(\mathbf{0})=\mathbf{0}$. Moreover, we assume~that~the~tensor $\SSS=(S_{ij})_{i,j=1,\dots,d}$ has $(p,\delta)$-structure, i.e.,
  for some $p \in (1, \infty)$, $ \delta\in [0,\infty)$, and the
  N-function $\varphi=\varphi_{p,\delta}$ (cf.~\eqref{eq:def_phi}), there
  exist constants $C_0, C_1 >0$ such that
   \begin{align}
       \sum\limits_{i,j,k,l=1}^d \partial_{kl} S_{ij} (\bfA)
       B_{ij}B_{kl} &\ge C_0 \, \frac{\phi'(|\bfA^{\textup{sym}}|)}{|\bfA^{\textup{sym}}|}\,|\bfB^{\textup{sym}}|^2\,,\label{assum:extra_stress.1}
       \\
       \big |\partial_{kl} S_{ij}({\bfA})\big | &\le C_1 \, \frac{\phi'(|\bfA^{\textup{sym}}|)}{|\bfA^{\textup{sym}}|}\label{assum:extra_stress.2}
   \end{align}
are satisfied for all $\bfA,\bfB \in \setR^{d\times d}$ with $\bfA^{\textup{sym}}\neq \mathbf{0}$ and all $i,j,k,l=1,\dots,d$.~The~constants $C_0,C_1>0$ and $p\in (1,\infty)$ are called the {characteristics} of $\SSS$.
\end{assumption}

\begin{remark}
  {\rm (i) It is well-known (cf.~\cite{dr-nafsa}) that the conditions
    \eqref{assum:extra_stress.1}, \eqref{assum:extra_stress.2} imply~the conditions in the definition of the $(p,\delta)$-structure in
    Part~I~of~the paper (cf.~\cite{kr-pnse-ldg-1}).

    (ii)     Assume that $\SSS$ satisfies Assumption \ref{assum:extra_stress} for some
    $\delta \in [0,\delta_0]$.~Then,~if~not otherwise stated, the
    constants in the estimates depend only on the characteristics~of~$\SSS$~and on $\delta_0\ge 0$, but are independent of $\delta\ge 0$.

    (iii)     
    Let $\phi$ and $\{\phi_a\}_{a\ge 0}$ be defined in \eqref{eq:def_phi} and \eqref{eq:phi_shifted}, respectively. 
    Then, the shifted operators 
    $\SSS_a\colon\mathbb{R}^{d\times d}\to \smash{\mathbb{R}_{\textup{sym}}^{d\times
      d}}$, $a \ge 0$, defined, for every $a \ge 0$
    and~$\bfA \in \mathbb{R}^{d\times d}$,~via 
\begin{align}
  \label{eq:flux}
  \SSS_a(\bfA) \coloneqq 
  \frac{\phi_a'(\abs{\bfA^{\textup{sym}}})}{\abs{\bfA^{\textup{sym}}}}\,
  \bfA^{\textup{sym}}\,, 
\end{align}
have $(p, \delta +a)$-structure.  In this case, the characteristics of
$\SSS_a$ depend~only~on~${p\in (1,\infty)}$ and are independent of
$\delta \geq 0$ and $a\ge 0$.
  }
\end{remark}

Closely related to the extra stress tensor $\SSS$ with
$(p,\delta)$-structure is the non-linear function
$\bfF\colon\setR^{d\times d}\to \setR^{d\times d}_{\textup{sym}}$, 
 for every $\bfA\in \mathbb{R}^{d\times d}$, defined via
\begin{align}
\begin{aligned}
    \bfF(\bfA)&\coloneqq (\delta+\vert \bfA^{\textup{sym}}\vert)^{\smash{\frac{p-2}{2}}}\bfA^{\textup{sym}}
   \,.
    \end{aligned}\label{eq:def_F}
\end{align}

The connections between
$\SSS,\bfF\colon\setR^{d \times d}
\hspace{-0.05em}\to\hspace{-0.05em} \setR^{d\times d}_{\textup{sym}}$ and
$\phi_a,(\phi_a)^*\colon\setR^{\ge
  0}\hspace{-0.05em}\to\hspace{-0.05em} \setR^{\ge
  0}$,~${a\hspace{-0.05em}\ge\hspace{-0.05em} 0}$, are best explained
by the following result (cf.~\cite{die-ett,dr-nafsa,dkrt-ldg}).

\begin{proposition}
  \label{lem:hammer}
  Let $\SSS$ satisfy Assumption~\ref{assum:extra_stress}, let $\varphi$ be defined in \eqref{eq:def_phi}, and let $\bfF,\bfF^*$ be defined in \eqref{eq:def_F}. Then, uniformly with respect to 
  $\bfA, \bfB \in \setR^{d \times d}$, we have that\vspace{-1mm}
    \begin{align}\label{eq:hammera}
        \begin{aligned}
        \big(\SSS(\bfA) - \SSS(\bfB)\big)
      :(\bfA-\bfB ) &\sim  \abs{ \bfF(\bfA) - \bfF(\bfB)}^2
      \\
      &\sim \phi_{\abs{\bfA^{\textup{sym}}}}(\abs{\bfA^{\textup{sym}}
        - \bfB^{\textup{sym}}})
      \\
      &\sim(\varphi_{\abs{\bfA^{\textup{sym}} }})^*(\abs{\SSS(\bfA ) - \SSS(\bfB )})
      \,.
      \end{aligned}
    \end{align}
  The constants in \eqref{eq:hammera} 
  depend only on the characteristics of ${\SSS}$.
\end{proposition} 
\begin{remark}\label{rem:sa}
  {\rm
For the operators $\SSS_a\colon\mathbb{R}^{d\times d}\hspace{-0.1em}\to\hspace{-0.1em}\smash{\mathbb{R}_{\textup{sym}}^{d\times
      d}}$, $a \ge  0$, defined~in~\eqref{eq:flux},~the~asser-tions of Proposition \ref{lem:hammer} hold with $\phi\colon\mathbb{R}^{\ge 0}\to \mathbb{R}^{\ge 0}$ replaced
by $\phi_a\colon\mathbb{R}^{\ge 0}\to \mathbb{R}^{\ge 0}$, $a\ge 0$.}
\end{remark}

The following  results can be found
in~\cite{DK08,dr-nafsa}.

\begin{lemma}[Change of Shift]\label{lem:shift-change}
    Let $\varphi$ be defined in \eqref{eq:def_phi} and let $\bfF$ be defined in \eqref{eq:def_F}. Then,
  for each $\varepsilon>0$, there exists $c_\varepsilon\geq 1$ (depending only
  on~$\varepsilon>0$ and the characteristics of $\phi$) such that for every $\bfA,\bfB\in\smash{\setR^{d \times d}_{\textup{sym}}}$ and $t\geq 0$, it holds
  \begin{align*}
    \smash{\phi_{\abs{\bfB}}(t)}&\leq \smash{c_\varepsilon\, \phi_{\abs{\bfA}}(t)
    +\varepsilon\, \abs{\bfF(\bfB) - \bfF(\bfA)}^2\,,}
    \\
        \smash{\phi_{\abs{\bfB}}(t)}&\leq \smash{c_\varepsilon\, \phi_{\abs{\bfA}} (t)
    +\varepsilon\, \phi_{\abs{\bfA}}\big(\bigabs{\abs{\bfB} - \abs{\bfA}}\big)\,,}
    \\
    \smash{(\phi_{\abs{\bfB}})^*(t)}&\leq \smash{c_\varepsilon\, (\phi_{\abs{\bfA}})^*(t)
                                      +\varepsilon\, \abs{\bfF(\bfB) - \bfF(\bfA)}^2}\,,
   \\
    \smash{(\phi_{\abs{\bfB}})^*(t)}&\leq \smash{c_\varepsilon\, (\phi_{\abs{\bfA}})^*(t)
    +\varepsilon\, \phi_{\abs{\bfA}}\big(\bigabs{\abs{\bfB} - \abs{\bfA}}\big)}\,.
  \end{align*}
\end{lemma}

\subsection{The $p$-Navier--Stokes system} 
Let us briefly recall some well-known facts about the $p$-Navier--Stokes system
 \eqref{eq:p-navier-stokes}. For $p\in (1,\infty)$, we define the function spaces
\begin{align*}
     \Vo\coloneqq (W^{1,p}_0(\Omega))^d\,,\qquad \Qo\coloneqq L_0^{p'}(\Omega)\coloneqq \big\{f\in L^{p'}(\Omega)\;|\;\mean{f}_{\Omega}=0\big\}\,.
\end{align*}
    With this particular notation, the weak formulation of problem \eqref{eq:p-navier-stokes} is the following:
    
\textbf{Problem (Q).} For given $\bfg\in L^{p'}(\Omega)$, find $(\bfv,q)\in \Vo \times \Qo$ such that  for all $(\bfz,z)^\top\in \Vo\times Q $, it holds
\begin{align}
    (\SSS(\bfD\bfv),\bfD\bfz)+([\nabla\bfv]\bfv,\bfz)-(q,\divo\bfz)&=(\bfg,\bfz)\label{eq:q1}\,,\\
    (\divo\bfv,z)&=0\label{eq:q2}\,.
\end{align}

Alternatively, we can reformulate Problem (Q) ``hidding'' the pressure.

\textbf{Problem (P).} \hspace{-0.1mm}For \hspace{-0.1mm}given \hspace{-0.1mm}$\bfg\!\in \!L^{p'}(\Omega)$, \hspace{-0.1mm}find \hspace{-0.1mm}$\bfv\!\in\! \Vo(0)$ 
\hspace{-0.1mm}such \hspace{-0.1mm}that \hspace{-0.1mm}for \hspace{-0.1mm}all \hspace{-0.1mm}${\bfz\!\in\! \Vo(0)}$,~\hspace{-0.1mm}it~\hspace{-0.1mm}holds
\begin{align}
    (\SSS(\bfD\bfv),\bfD\bfz)+([\nabla\bfv]\bfv,\bfz)&=(\bfg,\bfz)\,,\label{eq:p}
\end{align}
where $\Vo(0)\coloneqq \{\bfz\in \Vo\mid \divo \bfz=0\}$.

The theory of pseudo-monotone operators yields the existence of a weak
solution of Problem (P) for $p>\frac{3d}{d+2}$ (cf.~\cite{lions-quel}). DeRham's lemma, the solvability of the divergence
equation, and the negative norm theorem,~then, ensure the solvability
of Problem (Q).

\subsection{DG spaces, jumps and averages}\label{sec:dg-space}

\subsubsection{Triangulations}

\!\!We \hspace{-0.15mm}always \hspace{-0.15mm}denote \hspace{-0.15mm}by \hspace{-0.15mm}$\mathcal{T}_h$, \hspace{-0.15mm}$h\!>\!0$,\hspace{-0.15mm} a \hspace{-0.15mm}family \hspace{-0.15mm}of \hspace{-0.15mm}uniformly~\hspace{-0.15mm}shape
regular and conforming triangulations
of~${\Omega\hspace{-0.1em}\subseteq \hspace{-0.1em}\setR^d}$,~${d\hspace{-0.1em}\in\hspace{-0.1em} \set{2,3}}$,~cf.~\cite{BS08},~each
consisting of \mbox{$d$-dimensional} simplices $K$.  The parameter
$h\hspace{-0.1em}>\hspace{-0.1em}0$, refers to the~maximal~\mbox{mesh-size}~of~$\mathcal{T}_h$, for \hspace{-0.1mm}which
\hspace{-0.1mm}we \hspace{-0.1mm}assume \hspace{-0.1mm}for \hspace{-0.1mm}simplicity \hspace{-0.1mm}that \hspace{-0.1mm}$h \!\le\! 1$. \!Moreover, \hspace{-0.1mm}we \hspace{-0.1mm}assume~\hspace{-0.1mm}that~\hspace{-0.1mm}the~\hspace{-0.1mm}\mbox{chunkiness}~\hspace{-0.1mm}is bounded by some constant $\omega_0\!>\!0$, independent on
$h$. \!By $\Gamma_h^{i}$,~we~denote~the~interior~faces, and put
$\Gamma_h\coloneqq  \Gamma_h^{i}\cup \partial\Omega$.  We
assume~that~each~simplex~${K \in \mathcal{T}_h}$ has at most~one~face from $\partial\Omega$.  We introduce the following scalar
product~on~$\Gamma_h$
\begin{align*}
  \skp{f}{g}_{\Gamma_h} \coloneqq  \smash{\sum_{\gamma \in \Gamma_h} {\langle f, g\rangle_\gamma}}\,,\quad\text{ where }\quad\langle f, g\rangle_\gamma\coloneqq \int_\gamma f g \,\textup{d}s\quad\text{ for all }\gamma\in \Gamma_h\,,
\end{align*}
if all the integrals are well-defined. Similarly, we define the products 
$\smash{\skp{\cdot}{\cdot}_{\partial\Omega}}$ and~$\smash{\skp{\cdot}{\cdot}_{\Gamma_h^{i}}}$. We extend the notation of
modulars to the sets $\smash{\Gamma_h^{i}}$, 
$\partial \Omega$, and $\smash{\Gamma_h}$, i.e., we
define the modulars ${\rho_{\psi,B}(f)\coloneqq  \smash{\int _B
\psi(\abs{f})\,\textup{d}s}}$ for every $f \in \smash{L^\psi(B)}$, where $B= \smash{\Gamma_h^{i}}$~or~${B=\partial \Omega}$~or~${B=\smash{\Gamma_h}}$.

\subsubsection{Broken \hspace*{-0.1mm}function \hspace*{-0.1mm}spaces \hspace*{-0.1mm}and \hspace*{-0.1mm}projectors}

\!For every
\hspace*{-0.1mm}$m \hspace{-0.15em}\in\hspace{-0.15em}
\setN_0$~\hspace*{-0.1mm}and~\hspace*{-0.1mm}${K\hspace{-0.15em}\in\hspace{-0.15em} \mathcal{T}_h}$, we
denote by ${\mathcal P}_m(K)$, the space of polynomials of
degree at most $m$ on $K$. Then, for given $k \in \setN_0$ and $p\in
(1,\infty)$, 
we define the spaces
\begin{align}
  \begin{split}
    Q_h^k&\coloneqq \big\{ q_h\in L^1(\Omega) \fdg q_h|_K\in \mathcal{P}_k(K)\text{ for all }K\in \mathcal{T}_h\big\}\,,\\
    V_h^k&\coloneqq \big\{\bfv_h\in L^1(\Omega)^d\fdg \bfv_h|_K\in \mathcal{P}_k(K)^d\text{ for all }K\in \mathcal{T}_h\big\}\,,\\
    X_h^k&\coloneqq \big\{\bfX_h\in L^1(\Omega)^{d\times d}\fdg \bfX_h|_K\in \mathcal{P}_k(K)^{d\times d}\text{ for all }K\in \mathcal{T}_h\big\}\,,\\
        W^{1,p}(\mathcal T_h)&\coloneqq \big\{\bfw_h\in L^1(\Omega)^d\fdg \bfw_h|_K\in W^{1,p}(K)^d\text{ for all }K\in \mathcal{T}_h\big\}\,.
  \end{split}\label{eq:2.19}
\end{align}
In addition, for given $k \hspace{-0.1em}\in\hspace{-0.1em} \setN_0$, we set $\Qhkc\hspace{-0.1em}\coloneqq \hspace{-0.1em} Q_h^k\cap C^0(\overline{\Omega})$.
Note~that~${W^{1,p}(\Omega)\hspace{-0.1em}\subseteq\hspace{-0.1em} \WDG}$ and
$V_h^k\hspace{-0.1em}\subseteq \hspace{-0.1em}\WDG$. We denote by ${\PiDG \colon L^1(\Omega)\hspace{-0.1em}\to\hspace{-0.1em} V_h^k}$, the (local)
$L^2$-projection~into~$V_h^k$, which for every $\bfv \in
L^1(\Omega)$ and $\bfz_h
\in V_h^k$ is~defined~via  $\bighskp{\PiDG \bfv}{\bfz_h}=\hskp{\bfv}{\bfz_h}$. 
Analogously, we define the (local)
$L^2$-projection into $X_h^k$, i.e., ${\PiDG\colon L^1(\Omega)^{d\times d} \to \Xhk}$.

For every  $\bfw_h\in \WDG$, we denote by $\nabla_h \bfw_h\in L^p(\Omega)$,
the local gradient,~defined via
$(\nabla_h \bfw_h)|_K\coloneqq \nabla(\bfw_h|_K)$
for~all~${K\in\mathcal{T}_h}$.  For every $\bfw_h\in \WDG$ and
interior faces $\gamma\in \Gamma_h^{i}$ shared by adjacent elements
$K^-_\gamma, K^+_\gamma\in \mathcal{T}_h$, we~denote~by
\begin{align}
  \{\bfw_h\}_\gamma&\coloneqq \smash{\frac{1}{2}}\big(\textup{tr}_\gamma^{K^+}(\bfw_h)+
  \textup{tr}_\gamma^{K^-}(\bfw_h)\big)\in
  L^p(\gamma)\,, \label{2.20}\\
  \llbracket\bfw_h\otimes\bfn\rrbracket_\gamma
  &\coloneqq \textup{tr}_\gamma^{K^+}(\bfw_h)\otimes\bfn^+_\gamma+
    \textup{tr}_\gamma^{K^-}(\bfw_h)\otimes\bfn_\gamma^- 
    \in L^p(\gamma)\,,\label{eq:2.21}
\end{align}
the {average} and {normal jump}, resp., of $\bfw_h$ on $\gamma$.
Moreover,  for boundary faces $\gamma\in \partial\Omega$, we define boundary averages and 
boundary~jumps,~resp.,~via
\begin{align}
  \{\bfw_h\}_\gamma&\coloneqq \textup{tr}^\Omega_\gamma(\bfw_h) \in L^p(\gamma)\,,\label{eq:2.23a} \\
  \llbracket \bfw_h\otimes\bfn\rrbracket_\gamma&\coloneqq 
  \textup{tr}^\Omega_\gamma(\bfw_h)\otimes\bfn \in L^p(\gamma)\,,\label{eq:2.23} 
\end{align}
where $\bfn:\partial\Omega\to \mathbb{S}^{d-1}$ denotes the unit normal vector field to $\Omega$ pointing outward. 
Analogously, we
define $\{\bfX_h\}_\gamma$ and $ \llbracket\bfX_h\bfn\rrbracket_\gamma
$~for all $\bfX_h \in \Xhk$ and $\gamma\in \Gamma_h$. Furthermore, if there is no
danger~of~confusion, then we will omit the index $\gamma\in \Gamma_h$,~in~particular,~when we interpret jumps and averages as global functions defined on whole $\Gamma_h$.

\subsubsection{DG gradient and jump operators}

For every $k\in \mathbb{N}_0$ and  face $\gamma\in \Gamma_h$, we define the
\textbf{(local)~jump~\mbox{operator}}
$\smash{\boldsymbol{\mathcal{R}}_{h,\gamma}^k \colon\WDG \to X_h^k}$~for~every~${\bfw_h\in \smash{\WDG}}$ (using
Riesz representation)~via $(\boldsymbol{\mathcal{R}}_{h,\gamma}^k\bfw_h,\bfX_h)\coloneqq \langle \llbracket\bfw_h\otimes\bfn\rrbracket_\gamma,\{\bfX_h\}_\gamma\rangle_\gamma$ for all $\bfX_h\in X_h^k$. 
For every $k\in \mathbb{N}_0$, the \textbf{(global) jump operator} $\smash{\Rhk\coloneqq \sum_{\gamma\in \Gamma_h}{\boldsymbol{\mathcal{R}}_{\gamma,h}^k}\colon\WDG \to X_h^k}$, 
by definition, for every $\bfw_h\in \smash{\WDG}$ and  $\bfX_h\in X_h^k$ satisfies
\begin{align}
  \smash{\big(\Rhk\bfw_h,\bfX_h\big)=\big\langle
  \llbracket\bfw_h\otimes\bfn\rrbracket,\{\bfX_h\}\big\rangle_{\Gamma_h}\,.}\label{eq:2.25.1}
\end{align}
Then, \hspace*{-0.15mm}for \hspace*{-0.15mm}every \hspace*{-0.15mm}$k\!\in\! \mathbb{N}_0$, \hspace*{-0.15mm}the \hspace*{-0.15mm}\textbf{DG \hspace*{-0.15mm}gradient \hspace*{-0.15mm}operator} 
\hspace*{-0.15mm}$  {\Ghk\!\coloneqq \!\nabla_h\!-\!\Rhk\colon\WDG\!\to\! L^p(\Omega)}$,  
for every $\bfw_h\in \smash{\WDG}$ and $\bfX_h\in X_h^k$ satisfies
\begin{align}
\smash{\big(\Ghk\bfw_h,\bfX_h\big)=(\nabla_h\bfw_h,\bfX_h)
  -\big\langle \llbracket
  \bfw_h\otimes\bfn\rrbracket,\{\bfX_h\}\big\rangle_{\Gamma_h}
\,. } \label{eq:DGnablaR1}
\end{align}
Apart from that, for every $\bfw_h\in \smash{\WDG}$, we introduce the {DG norm} as
\begin{align}
    \smash{\|\bfw_h\|_{\nabla,p,h}\coloneqq \|\nabla_h\bfw_h\|_p+h^{\frac{1}{p}}\big\|h^{-1}\jump{\bfw_h\otimes \bfn}\big\|_{p,\Gamma_h}\,.}
\end{align}
There exists a constant $c\!>\!0$ (cf. \cite[(A.26)--(A.28)]{dkrt-ldg}) such that for every ${\bfw_h\!\in\! \smash{\WDG}}$, it holds
\begin{align}\label{eq:eqiv0}
    \smash{c^{-1}\,\|\bfw_h\|_{\nabla,p,h}\leq \big\|\Ghk\bfw_h\big\|_p+h^{\frac{1}{p}}\big\|h^{-1}\jump{\bfw_h\otimes \bfn}\big\|_{p,\Gamma_h}\leq c\,\|\bfw_h\|_{\nabla,p,h}\,.}
\end{align}
The following result extends the embedding results for classical
Sobolev spaces $W^{1,p}(\Omega)$ and broken polynomial spaces $\Vhk$
to DG Sobolev spaces $\WDG$.

\begin{proposition}\label{prop:emb}
    Let ${p,q\in[1,\infty)}$ be such that
    $W^{1,p}(\Omega) \vnor \vnor L^{q}(\Omega)$. If $p>q$, then we
    additionally assume
    that $h\sim h_K$ uniformly with respect to  $K \in \mathcal T_h$. Then, there
    exists a constant ${c=c(p,q,\omega_0)>0}$ such that for every $\bfw_h
    \in \WDG$, it holds
  \begin{align}\label{eq:emb}
    \|\bfw_h\|_q\le c\, \|\bfw_h\|_{\nabla , p,h}\,,
  \end{align}
  i.e., $\WDG \vnor L^q(\Omega)$.
\end{proposition}
\begin{proof}
  Note that $\Pia\colon L^1(\Omega)\to \Vhk$ satisfies
  \cite[Assumption A.1]{br-parabolic} (with $S_K$ replaced by $K$ and
  $r_0=0$). Therefore, proceeding (with some simplifications) as in the
  proof of \cite[Proposition A.2]{br-parabolic}, we deduce that there
    exists a constant ${c=c(p,q,\omega_0)>0}$ such that for every $\bfw_h
    \in \WDG$, it holds
  \begin{align}\label{eq:int-error1}
    \|\bfw_h -\Pia \bfw_h\|_q\le c\, h^{1 +d\min\{0, \frac 1q -\frac
    1p\}}\,  \|\nabla _h\bfw_h\|_{p}\,.
  \end{align}
  Using in \eqref{eq:int-error1} that $1 +d\,\smash{\min\{0, \frac 1q -\frac
  1p\}} \ge 0$, the discrete embedding \cite[Theorem 5.3]{ern-book} for
  functions from $\Vhk$, and the approximation properties of $\Pia$
  (cf.~\cite[Appendix~A.1]{dkrt-ldg}, \cite[Corollary A.8, Corollary A.19]{kr-phi-ldg}), we obtain
  \begin{align*}
    \|\bfw_h\|_q&\le \|\bfw_h-\Pia \bfw_h\|_q +\|\Pia \bfw_h\|_q
    \\
    &\le c\, \|\nabla _h\bfw_h\|_{p} +c\, \|\Pia \bfw_h\|_{p,\nabla ,h}
    \\
    &\le c\, \|\nabla _h\bfw_h\|_{p} +c\, \|\Pia \bfw_h -\bfw_h\|_{p,\nabla ,h}
      +c\, \| \bfw_h\|_{p,\nabla ,h}
    \\
    &\le c\, \| \bfw_h\|_{p,\nabla ,h}\,.
  \end{align*}
\end{proof}

For an  N-function $\psi$, we define the pseudo-modular\footnote{The
		definition of an pseudo-modular can be found in \cite{Mu}. We
	extend the notion of DG Sobolev spaces to  DG Sobolev-Orlicz
	spaces $W^{1,\psi}(\mathcal T_h)\coloneqq \big\{\bfw_h\!\in \!L^1(\Omega)\mid \bfw_h|_K\!\in\! W^{1,\psi}(K)\text{ for all }K\in \mathcal{T}_h\big\}$.} $\smash{m_{\psi,h}}\colon W^{1,\psi}(\mathcal T_h)\to \mathbb{R}^{\ge 0}$ for every $\bfw_h\in W^{1,\psi}(\mathcal T_h)$ via
  \begin{align} \label{def:mh}
    m_{\psi,h}(\bfw_h)&\coloneqq  h\,\rho_{\psi,\smash{\Gamma_h}}\big(h^{-1}\jump{\bfw_h\otimes \bfn}\big)\,.
  \end{align}
For $\psi = \phi_{p,0}$, we have that
$m_{\psi,h}(\bfw_h)=h\,\|h^{-1}\jump{\bfw_h\otimes
  \bfn}\|_{p,\Gamma_h}^p $~for~all~${\bfw_h\in W^{1,\psi}(\mathcal T_h)}$.
  
\subsubsection{Symmetric \hspace*{-0.15mm}DG \hspace*{-0.15mm}gradient \hspace*{-0.15mm}and \hspace*{-0.15mm}symmetric \hspace*{-0.15mm}jump \hspace*{-0.15mm}operators}

\hspace*{-0.15mm}For~\hspace*{-0.15mm}\mbox{every} ${\bfw_h\in \WDG}$, we denote by $\bfD_h\bfw_h\coloneqq [\nabla_h\bfw_h]^{\textup{sym}}\!\in\!
L^p(\Omega;\mathbb{R}^{d\times d}_{\textup{sym}})$ the~local~\mbox{symmetric}  gradient.  In addition, for every
$k\in \setN_0$ and
$\smash{X_h^{\smash{k,\textup{sym}}}\coloneqq X_h^k\cap
  L^p(\Omega;\mathbb{R}^{d\times d}_{\textup{sym}})}$,~we~define~the
{symmetric DG gradient
  operator}~$ \smash{\Dhk\colon\!\WDG\!\to\!
  L^p(\Omega;\mathbb{R}^{d\times
    d}_{\textup{sym}})}$,\hspace{.1em}for\hspace{.1em}every\hspace{.1em}${\bfw_h\!\in \!\WDG}$, via
$\smash{\Dhk\bfw_h\coloneqq [\Ghk\bfw_h]^{\textup{sym}}\!\in\!
  L^p(\Omega;\mathbb{R}^{d\times d}_{\textup{sym}})}$, i.e., for
every 
$\bfX_h\in \smash{X_h^{k,\textup{sym}}}$, we have
that
	\begin{align}
	\smash{	\big(\Dhk\bfw_h,\bfX_h\big)
		=(\bfD_h\bfw_h,\bfX_h)
		-\big\langle \llbracket \bfw_h\otimes\bfn\rrbracket,\{\bfX_h\}\big\rangle_{\Gamma_h}\,.}\label{eq:2.24}
	\end{align}
	Apart from that, for every $\bfw_h\in \WDG$, we introduce the {symmetric DG norm} as
	\begin{align}
	\smash{	\|\bfw_h\|_{\bfD,p,h}\coloneqq \|\bfD_h\bfw_h\|_p
		+\smash{h^{\frac{1}{p}}\big\|  h^{-1} \llbracket\bfw_h\otimes\bfn\rrbracket\big\|_{p,\Gamma_h}}\,.}\label{eq:2.29}
	\end{align}
	
	The following discrete Korn type inequalities play an important role
        in the numerical analysis of the
        $p$-Navier--Stokes~system~\eqref{eq:p-navier-stokes}.
	 \begin{proposition}[Discrete Korn inequality]\label{korn}
	    For every $p\in (1,\infty)$ and $k\in \setN$, there
		exists a constant~${c_{\mathbf{Korn}}>0}$ such that 
		for every $\bfv_h\in V_h^k$, it holds
		\begin{align}
		\smash{\|\bfv_h\|_{\nabla,p,h}\leq c_{\mathbf{Korn}}\,\|\bfv_h\|_{\bfD,p,h}}\label{eq:equi1}\,.
		\end{align}
	\end{proposition}

	\begin{proof}
          See \cite[Proposition 2.4]{kr-pnse-ldg-1}.
        \end{proof}
    \begin{proposition}[Korn type inequality]\label{prop:kornii}
      For every $p\in (1,\infty)$ and $k\in \setN$, there exists a
      constant~${c}>0$ such that for every
      $\bfv_h\in V_h^k$ and every  $\bfw_h\in W^{2,p}(\mathcal{T}_h)$, it holds
    \begin{align*}
        \smash{\|\bfw_h-\bfv_h\|_{\nabla,p,h}^p\leq c\,\|\bfw_h-\bfv_h\|_{\bfD,p,h}^p+c\,h^p\,\|\nabla_h^2\bfw_h\|_p^p\,.}
    \end{align*}
    \end{proposition}
    \begin{proof}
      See \cite[Proposition 2.8]{kr-pnse-ldg-2}.
      \end{proof}
	For the symmetric DG norm, there holds a similar relation like \eqref{eq:eqiv0}.
	
	\begin{proposition}\label{prop:equivalences}
          For every $p\in (1,\infty)$
          and $k\in \setN$, there exists
          a constant~${c>0}$
          such that for every $\bfw_h\in \WDG$, it holds
		\begin{align}
		\begin{aligned}
		\smash{c^{-1}\,\|\bfw_h\|_{\bfD,p,h}
		\leq \big\|\Dhk\bfw_h\big\|_p
		+h^{\frac{1}{p}}\big\|  h^{-1}\llbracket\bfw_h\otimes\bfn\rrbracket\big\|_{p,\Gamma_h}
		\leq c\,\|\bfw_h\|_{\bfD,p,h}\,.}
		\end{aligned}\label{eq:equi2}
		\end{align}
	\end{proposition}
	\begin{proof}
		See \cite[Proposition 2.5]{kr-pnse-ldg-1}.
	\end{proof}
	
\subsubsection{\!\!DG \hspace*{-0.1mm}divergence \hspace*{-0.15mm}operator}
	\!\!For \hspace{-0.15mm}every \hspace{-0.15mm}$\bfw_h\!\in\!\WDG$, \hspace{-0.15mm}we \hspace{-0.15mm}denote~\hspace{-0.15mm}by~\hspace{-0.15mm}$\textup{div}_h\bfw_h$ $\coloneqq  \text{tr}(\nabla_h \bfw_h)\in L^p(\Omega)$, the {local divergence}. 
	In addition, for every $k\in \setN_0$,~the~{DG~divergence~operator} 
	$\Divhk\colon\WDG\to L^p(\Omega)$, for every $\bfw_h\in \WDG$, is~defined~via $\smash{\Divhk\bfw_h\coloneqq \text{tr}(\Ghk\bfw_h)
		=\text{tr}(\Dhk\bfw_h)\in Q_h^k}$,~i.e., 
	for every 
	$z_h  \in \smash{Q_h^k}$, we~have~that
	\begin{align*}
	    \smash{\big(\Divhk\bfw_h,z_h\big)
	    =(\textup{div}_h\bfw_h,z_h)
		-\langle \llbracket \bfw_h\cdot\bfn\rrbracket,\{z_h\}\rangle_{\Gamma_h}\,.}
	\end{align*}
    Therefore,  for every $\bfv\in W^{1,p}_0(\Omega)$~and~$z_h\in \Qhkc$, we have that
\begin{align}
  \label{eq:div-dg}
  \begin{aligned}
    \smash{\big(\Divhk \PiDG \bfv,z_h\big)=-( \bfv,\nabla z_h)=(\divo\bfv, z_h)\,.}
\end{aligned}
\end{align}

\section{Fluxes and LDG formulations}\label{sec:ldg}
	
To obtain the LDG formulation of~\eqref{eq:p-navier-stokes}~for
$k \in \setN$, we proceed as in Part I \cite[Sec. 3]{kr-pnse-ldg-1} to get
the discrete counterpart~of~\mbox{Problem}~(Q). 
   Recall that, restricting ourselves to the case that $q_h\in \Qhkco\coloneqq\Qhko\cap C^0(\overline{\Omega})$,~the~numerical fluxes are, for every stabilization
   parameter $\alpha>0$, defined~via
   \begin{align}
   	\label{def:flux-v1}
   	\flux{\bfv}_{h,\sigma}(\bfv_h) &\coloneqq  
   	\begin{cases}
   		\avg{\bfv_h} &\text{on $\Gamma_h^{i}$}
   		\\
   		\bfv^* &\text{on $\partial\Omega$}
   	\end{cases}\,,
   	\quad \flux{\bfv}_{h,q}(\bfv_h) \coloneqq  
   	\begin{cases}
   		\avg{\bfv_h} &\text{on $\Gamma_h^{i}$}\\
   		\bfv^* &\text{on $\partial\Omega$}
   	\end{cases}\,,\\[-0.25mm]
   	\label{def:flux-q}
   	\flux{q}(q_h) &\coloneqq 
   	q_h 
   	\quad \text{on $\Gamma_h$}\,,\\[-0.25mm]
   	\label{def:flux-S}
   	\flux{\bfS}(\bfv_h, \bfS_h,\bfL_h) &\coloneqq 
   	\avg{\bfS_h} \hspace*{-0.1em}- \hspace*{-0.1em}\alpha\, \SSS_{\smash{\sss}}\big(h^{-1}\jump{(\bfv_h     \hspace*{-0.1em}-\hspace*{-0.1em} \bfv_0^*)\hspace*{-0.1em}\otimes\hspace*{-0.1em} \bfn}\big)
   	\quad \text{on $\Gamma_h$}\,,\\[-0.25mm]
   	\label{def:flux-K}
   	\flux{\bfK}(\bfv_h) &\coloneqq 
   	\avg{\bfK_h} 
   	\quad \text{on $\Gamma_h$}\,,\\[-0.25mm]
   	\label{def:flux-F}
   	\flux{\bfG}(\PiDG\bfG) &\coloneqq  
   	\bigavg{\PiDG \bfG} \quad\text{on $\Gamma_h$}\,,
   \end{align}
   where the operator $\SSS_{\smash{\sss}}$ is defined as in
   \eqref{eq:flux}.  Thus we arrive~at~an~inf-sup stable system
   without using a pressure stabilization. It is also possible
     to work with a
   discontinuous pressure. In this case one has to modify the fluxes as follows:
   $\flux{q}(q_h) \coloneqq  \avg{q_h} $~on~$\Gamma_h$~and
   $\flux{\bfv}_{h,q}(\bfv_h) \coloneqq 
   \avg{\bfv_h}+h\jump{q_h\bfn}$ on $\Gamma_h^{i}$,
   $\flux{\bfv}_{h,q}(\bfv_h) \coloneqq  \bfv^* $ on
   $\partial\Omega$.

    As in Part I of the paper (cf.~\cite{kr-pnse-ldg-1}), we arrive at the {flux
      formulation} of~\eqref{eq:p-navier-stokes}, which reads:
    For given $\bfg\in \smash{L^{p'}(\Omega)} $, find
    $(\bfL_h,\bfS_h,\bfK_h,\bfv_h,q_h)^\top \in \Xhk \times
    \Xhk\times\Xhk\times\Vhk \times \Qhkco$ such that for all
    $(\bfX_h,\bfY_h,\bfZ_h,\bfz_h,z_h)^\top$
    $ \in \Xhk \times \Xhk\times\Xhk\times\Vhk \times \Qhkc$, it holds
    \begin{align}
        \hskp{\bfL_h}{\bfX_h} &= \bighskp{\Ghk \bfv_h }{
          \bfX_h}\,,  \notag
        \\
        \hskp{\bfS_h}{\bfY_h} &= \hskp{\SSS(\bfL_h^{\textup{sym}})}{
          \bfY_h}\,, \notag
           \\
       \label{eq:DG} \hskp{\bfK_h}{\bfZ_h} &= \hskp{\bfv_h\otimes \bfv_h}{
          \bfZ_h}\,, 
        \\[-0.5mm]
        \bighskp{\bfS_h-\tfrac{1}{2}\bfK_h-q_h\mathbf{I}_d}{\Dhk \bfz_h} &=
        \hskp{\bfg-\tfrac{1}{2}\bfL_h\,\bfv_h}{\bfz_h} \notag
        \\[-0.5mm]
        &\quad - \alpha \,\bigskp{\SSS_{\smash{\avg{\abs{\Pia\bfL_h^{\textup{sym}}}}}}(h^{-1} \jump{\bfv_h \otimes
            \bfn})}{ \jump{\bfz_h \otimes \bfn}}_{\Gamma_h}\,, \notag\\[-1.5mm]
        \bighskp{\Divhk \bfv_h}{z_h}&=0\,.  \notag
    \end{align}  
    
    Now, we eliminate in the system~\eqref{eq:DG} the
    variables $ \bfL_h\hspace{-0.15em}\in\hspace{-0.15em} \Xhk$, $ \bfS_h\hspace{-0.15em}\in\hspace{-0.15em}     \Xhk$~and~${\bfK_h\hspace{-0.15em}\in\hspace{-0.15em} \Xhk}$ to derive a system only expressed in
    terms of the two variables ${\bfv_h\in\Vhk}$~and~${q_h\in
      \Qhkco}$. To this end, we observe that it follows from \eqref{eq:DG}$_{1,2,3}$ that
      $\bfL_h= \Ghk \bfv_h + \Rhk\bfv^*$, $\bfL_h^{\textup{sym}}=\Dhk\bfv_h$,  $\bfS_h=\PiDG\SSS(\bfL_h^{\textup{sym}})$, $\bfK_h= \PiDG
      (\bfv_h \otimes \bfv_h)$.
    If we insert this into \eqref{eq:DG}$_4$,
    we get the discrete counterpart~of~Problem~(Q):
    
    \textbf{Problem (Q$_h$).} For given $\bfg\in L^{p'}(\Omega)$, find $(\bfv_h,q_h)^\top\!\in\! \Vhk\times \Qhkco$ such~that~for~all $ (\bfz_h,z_h)^\top \in \Vhk\times\Qhkco$, it holds
    \begin{align}
        \bighskp{\SSS(\Dhk \bfv_h)-\tfrac{1}{2}\bfv_h\otimes \bfv_h-q_h\mathbf{I}_d}{\Dhk
          \bfz_h}\label{eq:primal1}
        &=     \bighskp{\bfg-\tfrac{1}{2}[\Ghk \bfv_h]\bfv_h}{\bfz_h}
        \\[-0.5mm]
        &\quad - \alpha \big\langle\SSS_{\smash{\sssl}}(h^{-1} \jump{\bfv_h\otimes\notag
            \bfn}), \jump{\bfz_h \otimes \bfn}\big\rangle_{\Gamma_h}\!,\\[-1.5mm]
      \bighskp{\Divhk \bfv_h}{z_h}&=0 \notag\,.
    \end{align}   
    Next, we eliminate in the system~\eqref{eq:primal1}, the
    variable $q_h\!\in\! \Qhkco$ to derive~a~system~only~expressed in terms of the single variable $\bfv_h\in \Vhk$. To this end, we introduce the space
    \begin{align*}
        V_h^k(0)\coloneqq \big\{\bfv_h\in V_h^k\mid \bighskp{\Divhk \bfv_h}{z_h}=0\textup{ for all }z_h\in     \Qhkc\big\}\,.
    \end{align*}
    Consequently, since $\hskp{z_h\mathbf{I}_d}{\Dhk
          \bfz_h}=\hskp{z_h}{\Divhk
          \bfz_h}=0$ for all $z_h\in \Qhkc$ and $\bfz_h\in V_h^k(0)$,
    we~get~the~discrete counterpart of Problem (P):
    
    \textbf{Problem (P$_h$).} For given $\bfg\!\in\! L^{p'}(\Omega)$, find $\bfv_h\!\in\! V_h^k(0)$ such that~for~all~${\bfz_h\!\in\! V_h^k(0)}$, it holds
    \begin{align} 
          \bighskp{\SSS(\Dhk \bfv_h)-\tfrac{1}{2}\bfv_h\otimes \bfv_h}{\Dhk
          \bfz_h}&=     \bighskp{\bfg-\tfrac{1}{2}[\Ghk \bfv_h]\bfv_h}{\bfz_h}\label{eq:primal2}\\[-0.5mm]
       &\quad - \alpha \big\langle\SSS_{\smash{\sssl}}(h^{-1} \jump{\bfv_h\otimes
            \bfn}), \jump{\bfz_h \otimes \bfn}\big\rangle_{\Gamma_h}\,.\notag
    \end{align}
    Problem (Q$_h$) and Problem (P$_h$) are called {primal formulations} of
    the system \eqref{eq:p-navier-stokes}. 

    Well-posedness (i.e., solvability), stability (i.e., a priori
    estimates), and (weak) convergence of Problem (Q$_h$) and Problem
    (P$_h$) are proved in Part~I~of the paper~(cf.~\cite{kr-pnse-ldg-1}).
    
\section{Convergence rates for the pressure}\label{sec:rates}

Let us start with the main result of this paper:
\begin{theorem}
  \label{thm:error}
  Let $\SSS$ satisfy Assumption~\ref{assum:extra_stress} with
  $p\in(2,\infty)$ and $\delta> 0$,~let~$k\in \mathbb{N}$, and let
  $\bfg\in L^{p'}(\Omega)$. Moreover, let
  $(\bfv,q)^\top \in \Vo(0)\times \Qo$ be a solution of
  Problem~(Q) (cf.~\eqref{eq:q1}, \eqref{eq:q2}) with
  $\bfF(\bfD \bfv) \in W^{1,2}(\Omega)$ 
  and let $(\bfv_h,q_h)^\top \in \Vhk(0)\times \Qhkco$ be a solution
  of Problem (Q$_h$) (cf.~\eqref{eq:primal1}) for $\alpha>0$.   Then,
  there exists a constant $c_0 >0$, depending only on the characteristics of
  $\SSS$, $\delta^{-1}$, 
  $\omega_0$, 
  $\alpha^{-1}$,~and~${k}$,~such~that if $\norm{\nabla
    \bfv}_2\le c_0$, then, it holds
  \begin{align*}
    \|q_h-q\|_{p'}\leq c\, h 
    +c\,\big(\rho_{\smash{(\varphi_{\vert \bfD\bfv\vert})^*,\Omega}}(h\,\nabla q)\big)^{\smash{\frac{1}{2}}} 
  \end{align*}
  with a constant $c>$ depending only on the characteristics of
  $\SSS$, $\norm{\bfF(\bfD\bfv)}_{1,2}$, $\norm{\nabla q}_{p'}$,
  $\norm{\bfg}_{p'}$, 
  $\delta^p\abs{\Omega}$, $\delta^{-1}$, 
  $\omega_0$, 
  $\alpha^{-1}$,  $k$, 
  and $c_0$.
\end{theorem}


\begin{corollary}\label{cor:error}
  Let the assumptions of Theorem \ref{thm:error} be satisfied. Then,~it~holds
  \begin{align}\label{eq:p'1}
    \|q_h-q\|_{p'}\leq 
    c\,h^{\smash{\frac{p'}{2}}} 
  \end{align}
  with a constant $c>0$ depending only on the characteristics of
  $\SSS$, $\norm{\bfF(\bfD\bfv)}_{1,2}$,~$\norm{\nabla q}_{p'}$, $\norm{\bfg}_{p'}$, 
  $\delta^p\abs{\Omega}$, $\delta^{-1}>0$, 
  $\omega_0$, 
  $\alpha^{-1}$,  $k$, 
  and $c_0$.
  If, in addition, $\bfg\in L^2(\Omega)$, then
  \begin{align}\label{eq:p'a1}
    \|q_h-q\|_{p'}\leq c\, h 
  \end{align}
  with a constant $c>0$ depending only on the characteristics of
  $\SSS$, $\norm{\bfF(\bfD\bfv)}_{1,2}$,  $\norm{\bfg}_{p'}$, $\norm{(\delta+\vert
    \bfD\bfv\vert)^{\smash{\frac{2-p}{2}}}\nabla q}_2$, 
  $\delta^p\abs{\Omega}$, $\delta^{-1}$, 
  ${\omega_0}$, 
  $\alpha^{-1}$,  $k$, 
  and $c_0$.
\end{corollary}

  \begin{remark}\label{rem:best}
    {\rm
   (i) Note that, in view of \cite[Lemma 2.6]{kr-pnse-ldg-2}, in Theorem \ref{thm:error} the assumption
    $\bfg\in L^{p'}(\Omega)$ is equivalent to
    $\nabla q\in L^{p'}(\Omega)$ and in Corollary \ref{cor:error}, the
    assumption $\bfg\in L^2(\Omega)$ can be replaced by
    $(\delta+\abs{\bfD\bfv})^{2-p}\vert \nabla q\vert^2\in
    L^1(\Omega)$.

    (ii) To the best of the authors knowledge the results in Theorem
    \ref{thm:error} and Corollary~\ref{cor:error} are the first
    convergence results for the pressure for the $p$-Navier--Stokes
    equations~\eqref{eq:p-navier-stokes} both for DG methods and FE
    methods.}
  \end{remark}

  Due to the modified numerical flux \eqref{def:flux-S}, the error
analysis can no longer be performed in terms of modulars as in
\cite{bdr-phi-stokes}, but in terms of Luxembourg
norms~only. The reason for that is that estimates in modulars
  are usually proved by the additive Young's inequality, while for
  estimates in norms we can use the multiplicative H\"older's
  inequality. Thus, certain terms which can not be absorbed in the
  modular appear as a bounded factor in the corresponding norm estimate.
Since all error estimates proved in \cite{kr-pnse-ldg-2} are
formulated in terms of modulars, we need to translate them in terms of
Luxembourg norms. The following lemma helps~us~to~do~this.

\begin{lemma}\label{lem:just_to_mighty}
    Let $\phi\colon\mathbb{R}^{\ge 0}\to \mathbb{R}^{\ge 0}$ be defined by \eqref{eq:def_phi} for $p\in [2,\infty)$ and $\delta\ge 0$ and let $a\in L^p(\Omega)$ with $a\ge 0$ a.e. in $\Omega$. Then, the following statements apply:
    \begin{itemize}
        \item[(i)] For every $c_0\ge 1$, $\gamma\leq 1$, and $f\in L^p(\Omega)$ from $\rho_{\varphi_a,\Omega}(f)\leq c_0\,\gamma$, 
    it follows~that $\smash{\|f\|_{\varphi_a}\leq (2c_0)^{\frac{1}{2}}\,\Delta_2(\varphi)\, \gamma^{\frac{1}{p}}}$.
    
    \item[(ii)] For every $c_0\ge 1$, $\gamma\leq 1$, and $g\in L^{p'}(\Omega)$ from $\rho_{(\varphi_a)^*,\Omega}(g)\leq c_0\, \gamma$, 
    it follows~that $\smash{\|g\|_{(\varphi_a)^*}\leq (c_0\,c(p))^{\smash{\frac{1}{p'}}} \gamma^{\frac{1}{2}}}$.
    \end{itemize}
\end{lemma}

\begin{proof}
    \textit{ad (i)}. Observing that, owing to \cite[Lemma 5.1, (5.11)]{dr-nafsa}, \eqref{eq:phi_shifted}, \eqref{eq:def_phi}, $p\ge 2$, the $\Delta_2$-condition of $\varphi_{a(x)}\colon\mathbb{R}^{\ge 0}\hspace{-0.15em}\to\hspace{-0.15em} \mathbb{R}^{\ge 0}$ for a.e. $x\hspace{-0.15em}\in\hspace{-0.15em} \Omega$ and \cite[Lemma~5.3]{dr-nafsa},~for~every~${\lambda \hspace{-0.15em}\leq\hspace{-0.15em} 1}$, $c\ge 1$,  $t\ge 0$ and for a.e. $x\in \Omega$, it holds
    \begin{align*}
        \varphi_{a(x)}\Big(\frac{t}{c\lambda}\Big)
        &\leq
        (\varphi_{a(x)})'\Big(\frac{t}{c\lambda}\Big)\frac{t}{c\lambda}
        =
        \Big(\delta+a(x)+\frac{t}{c\lambda}\Big)^{p-2}\Big(\frac{t}{c\lambda}\Big)^2
        \\
        &\leq 
        \frac{1}{\lambda^pc^2}(\delta+a(x)+t)^{p-2}t^2
        =
        \frac{1}{\lambda^pc^2}(\varphi_{a(x)})'(t)t
         \\
        &\leq 
        \frac{1}{\lambda^pc^2}\varphi_{a(x)}(2t)
        \leq 
        \frac{\Delta_2(\varphi_{a(x)})}{\lambda^pc^2}\varphi_{a(x)}(t)\leq 
        \frac{2\Delta_2(\varphi)^2}{\lambda^pc^2}\varphi_{a(x)}(t)\,,
    \end{align*}
    choosing $\lambda = \gamma^{\frac{1}{p}}$ and $c=(2c_0)^{\frac{1}{2}}\,\Delta_2(\varphi)$, we find that
    \begin{align*}
        \rho_{\varphi_a,\Omega}\bigg(\frac{f}{(c_0\,\Delta_2(\varphi))^{\frac{1}{2}} \gamma^{\frac{1}{p}}}\bigg)\leq \frac{1}{c_0 \gamma}\rho_{\varphi_a,\Omega}(f)\leq 1\,,
    \end{align*}
    so that, from the definition of the Luxembourg norm, we conclude the assertion.
    
    \textit{ad (ii)}. Observing that, due to $p\ge 2$, for every $\lambda \leq 1$, $c\ge 1$, and $t\ge 0$,~it~holds
    \begin{align*}
        (\varphi_{a(x)})^*\Big(\frac{t}{c\lambda}\Big)
        &\leq
        c(p)
        \Big(\delta^{p-1}+a(x)^{p-1}+\frac{t}{c\lambda}\Big)^{p'-2}\Big(\frac{t}{c\lambda}\Big)^2
        \\
        &\leq 
        \frac{c(p)}{\lambda^2}
        \Big(\frac{\delta^{p-1}}{c}+\frac{a(x)^{p-1}}{c}+\frac{t}{c}\Big)^{p'-2}\Big(\frac{t}{c}\Big)^2
        \\
        &\leq 
        \frac{c(p)}{\lambda^2c^{p'}}
        (\delta^{p-1}+a(x)^{p-1}+t)^{p'-2}t^2
        \\
        &
        \le   \frac{c(p)}{\lambda^2c^{p'}}(\varphi_{a(x)})^*(t)\,,
    \end{align*}
    where $c(p)>0$ depends only on $p\ge 2$,
    choosing $\lambda = \gamma^{\frac{1}{2}}$ and $c=(c_0\,c(p))^{\frac{1}{p'}}$,~we~find~that
    \begin{align*}
        \rho_{(\varphi_a)^*,\Omega}\bigg(\frac{g}{(c_0\,c(p))^{\smash{\frac{1}{p'}}} \gamma^{\frac{1}{2}}}\bigg)\leq \frac{1}{c_0\gamma
        }\rho_{(\varphi_a)^*,\Omega}(g)\leq 1\,,
    \end{align*}
    so that, from the definition of the Luxembourg norm, we conclude the assertion.
\end{proof}

In order to prove the results in Theorem \ref{thm:error} and Corollary \ref{cor:error}, we need to derive a system similar~to~\eqref{eq:primal1},~which~is satisfied by a solution
of our original problem \eqref{eq:p-navier-stokes}.
Using the notation $\bfL=\nabla\bfv $, $\bfS =\SSS(\bfL^{\textup{sym}})$, $\bfK =\bfv\otimes \bfv$,
we find that
${(\bfv, \bfL, \bfS,\bfK)^\top \!\in\! W^{1,p}(\Omega)\times
L^p(\Omega)\times L^{p'}(\Omega)\times L^{p'}(\Omega)}$. If, in
addition,~${\bfS,\bfK,q \!\in\!
W^{1,1}(\Omega)}$, we observe as in \cite{dkrt-ldg}, i.e., using
\mbox{integration-by-parts}, the
\hspace{-0.1mm}projection \hspace{-0.1mm}property \hspace{-0.1mm}of \hspace{-0.1mm}$\PiDG$, \hspace{-0.1mm}the \hspace{-0.1mm}definition \hspace{-0.1mm}of \hspace{-0.1mm}the \hspace{-0.1mm}discrete
\hspace{-0.1mm}gradient \hspace{-0.1mm}and~\hspace{-0.1mm}jump~\hspace{-0.1mm}\mbox{functional}, that
\begin{align}\label{eq:cont}
  \begin{aligned}
    \hskp{\bfL}{\bfX_h} &= \hskp{\nabla\bfv}{ \bfX_h}\,,
    \\
    \hskp{\bfS}{\bfY_h} &= \bighskp{\SSS(\bfL^{\textup{sym}})}{ \bfY_h}\,,
    \\
    \hskp{\bfK}{\bfZ_h} &= \bighskp{\bfv\otimes \bfv}{ \bfZ_h}\,,
    \\
    \bighskp{\bfS-\tfrac{1}{2}\bfK-q\mathbf{I}_d}{\Dhk \bfz_h} &= \hskp{\bfg-\tfrac{1}{2}\bfL\bfv}{\bfz_h} +
      \bigskp{\avg{\bfS}-\bigavg{\PiDG\bfS}}{\jump{\bfz_h\otimes
          \bfn}}_{\Gamma_h}\\&\quad+
          \tfrac{1}{2}\bigskp{\bigavg{\PiDG\bfK}-\avg{\bfK}}{\jump{\bfz_h\otimes
          \bfn}}_{\Gamma_h}
          \\&\quad+
          \bigskp{\bigavg{\PiDG(q\mathbf{I}_d)}-\avg{q\mathbf{I}_d}}{\jump{\bfz_h\otimes
          \bfn}}_{\Gamma_h}\,,
    \\
    \hskp{\Divhk\bfv}{\bfz_h} &= 0
            \end{aligned}
\end{align}  
is satisfied for all $(\bfX_h,\bfY_h,\bfZ_h,\bfz_h)^\top \in \Xhk \times \Xhk\times \Xhk \times \Vhk$.
As a result, using \eqref{eq:cont}, \eqref{eq:primal1} and
\eqref{eq:div-dg}, we~arrive~at
\begin{align}
     &\bighskp{\SSS(\Dhk \bfv_h) - \SSS(\bfD
      \bfv)}{\Dhk \bfz_h}
    + \alpha \big\langle\SSS_{\smash{\sssl}}(h^{-1} \jump{\bfv_h\otimes
        \bfn}), \jump{\bfz_h \otimes \bfn}\big\rangle_{\Gamma_h}\label{eq:errorprimal}\\
    &=\hskp{q_h-q}{\Divhk \bfz_h}+b_h(\bfv,\bfv,\bfz_h)-b_h(\bfv_h,\bfv_h,\bfz_h)\notag
    +\bigskp{\avg{\bfS}-\bigavg{\PiDG\bfS}}{\jump{\bfz_h\otimes
          \bfn}}_{\Gamma_h}
          \\&\quad+
          \tfrac{1}{2}\bigskp{\bigavg{\PiDG\bfK}-\avg{\bfK}}{\jump{\bfz_h\otimes
          \bfn}}_{\Gamma_h}\notag +
          \bigskp{\bigavg{\PiDG(q\mathbf{I}_d)}-\avg{q\mathbf{I}_d}}{\jump{\bfz_h\otimes
          \bfn}}_{\Gamma_h}\,,\notag
\end{align}
which is satisfied for all $(\bfz_h,z_h)^\top \in \Vhk\times
\Qhkc$. Here, we denoted~the~discrete~convective term by
$b_h\colon\WDG\times\WDG\times\WDG \to  \mathbb{R}$, which  is defined via
\begin{align*}
    b_h(\bfx_h,\bfy_h,\bfz_h)\coloneqq \tfrac{1}{2}\hskp{\bfz_h\otimes \bfx_h}{\Ghk\bfy_h}-\tfrac{1}{2}\hskp{\bfy_h\otimes \bfx_h}{\Ghk\bfz_h}
\end{align*}
for all $(\bfx_h,\bfy_h,\bfz_h)^\top\in \WDG\times\WDG\times\WDG$.
    
Now we have prepared everything to prove our main result Theorem \ref{thm:error}.    
\begin{proof}[Proof of Theorem \ref{thm:error}] 
  \!From our assumptions, resorting to \cite[Lemma 2.6]{kr-pnse-ldg-2},~\mbox{follows}
  that $\smash{\nabla q \in L^{p'}(\Omega)}$, which together with $\smash{(\varphi_{\vert \bfD\bfv\vert})^*(t)\leq
c\,t^{p'}}$, valid for every $t\ge 0$, yields
  $\smash{\rho_{(\varphi_{\vert \bfD\bfv\vert})^*,\Omega}(\nabla
  q)\le \norm{\nabla q}_{p'}^{p'}<\infty}$. Moreover, appealing to \cite[Lemma~6.10]{ern-book}, we
  deduce the existence of a constant $\beta>0$ such
  that~for~every~${z_h\in \Qhkco}$, it holds the~LBB~condition
    \begin{align}
    \label{thm:error.1}
        \beta\|z_h\|_{p'}\leq \sup_{\bfz_h\in \Vhk;{\|\bfz_h\|_{\nabla,p,h}\leq 1}}{( z_h,\Divhk\bfz_h)}\,.
    \end{align}
    On the other hand, due to \eqref{eq:errorprimal}, for every $\bfz_h\in \Vhk$, we have that
    \begin{align}\label{thm:error.2}
  \begin{aligned}
    \hskp{q_h-q}{\Divhk \bfz_h}
     &
     =\bighskp{\SSS(\Dhk \bfv_h) - \SSS(\bfD
      \bfv)}{\Dhk \bfz_h}
    \\&
    \quad
    + \alpha \big\langle\SSS_{\smash{\sssl}}(h^{-1} \jump{\bfv_h\otimes
        \bfn}), \jump{\bfz_h \otimes \bfn}\big\rangle_{\Gamma_h}
    \\&
    \quad
    -b_h(\bfv,\bfv,\bfz_h)+b_h(\bfv_h,\bfv_h,\bfz_h)
    \\&
    \quad
    +\bigskp{\bigavg{\PiDG\bfS}-\avg{\bfS}}{\jump{\bfz_h\otimes
          \bfn}}_{\Gamma_h}
    \\&
    \quad
    +\tfrac{1}{2}\bigskp{\avg{\bfK}-\bigavg{\PiDG\bfK}}{\jump{\bfz_h\otimes\bfn}}_{\Gamma_h}
    \\&
    \quad
    +\bigskp{\avg{q\mathbf{I}_d}-\bigavg{\PiDG(q\mathbf{I}_d)}}{\jump{\bfz_h\otimes\bfn}}_{\Gamma_h}
    \\&=I_1+\alpha\, I_2+I_3+\dots+I_6\,.
  \end{aligned}
\end{align}

So, let us next estimate $I_1,\dots, I_6$ for some arbitrary $\bfz_h\in \Vhk$ with $\|\bfz_h\|_{\nabla,p,h}\leq 1$:

(\textit{ad $I_1$}).
Using the generalized Hölder inequality \eqref{eq:gen_hoelder}, we find that
\begin{align}\label{thm:error.3}
    \begin{aligned}
        \vert I_1\vert &\leq 2\,\bignorm{\bfS(\Dhk\bfv_h)-\bfS(\bfD\bfv)}_{(\varphi_{\vert     \bfD\bfv\vert})^*}\bignorm{\Dhk\bfz_h}_{\varphi_{\vert \bfD\bfv\vert}}\\&\eqqcolon  2\,I_{1,1}\cdot I_{1,2}\,.
    \end{aligned}
\end{align}
Appealing to \eqref{eq:hammera} and \cite[Theorem 4.1]{kr-pnse-ldg-2}, we have that 
\begin{align}\label{thm:error.4}
    \begin{aligned}
    \rho_{(\varphi_{\vert \bfD\bfv\vert})^*,\Omega}\big(\bfS(\Dhk\bfv_h)-\bfS(\bfD\bfv)\big)&\leq c\,\bignorm{\bfF(\Dhk\bfv_h)-\bfF(\bfD\bfv)}_2^2
    \\&
    \leq c\,h^2\|\bfF(\bfD\bfv)\|_{1,2}^2+c\,\rho_{(\varphi_{\vert \bfD\bfv\vert})^*,\Omega}(h\,\nabla q)\,.
    \end{aligned}
\end{align}
Since, by assumption, we have that $h\leq 1$, for 
\begin{align*}
    c_0&\coloneqq \max\big\{1,c\,\|\bfF(\bfD\bfv)\|_{1,2}^2+c\, \rho_{(\varphi_{\vert \bfD\bfv\vert})^*,\Omega}(\nabla q)\big\}\ge 1\,,\\ \gamma&\coloneqq c_0^{-1}\big(c\,h^2\|\bfF(\bfD\bfv)\|_{1,2}^2+c\,\rho_{(\varphi_{\vert \bfD\bfv\vert})^*,\Omega}(h\,\nabla q)\big)\leq 1\,,
\end{align*}
Lemma \ref{lem:just_to_mighty} yields a constant $c(p)>0$, depending only on $p\in (2,\infty)$, such that
\begin{align}\label{thm:error.5}
    \begin{aligned}
    I_{1,1}\leq c_0^{\frac{1}{p'}-\frac{1}{2}}c(p)^{\frac{1}{p'}}\big(c\,h^2\|\bfF(\bfD\bfv)\|_{1,2}^2+c\,\rho_{(\varphi_{\vert \bfD\bfv\vert})^*,\Omega}(h\,\nabla q)\big)^{\frac{1}{2}}\,.
    \end{aligned}
\end{align}
Using the shift change in Lemma \ref{lem:shift-change}, that $\varphi(t)\leq c(p)\,(\delta^p+t^p)$ for all $t\ge 0$, and \eqref{eq:equi2} we find that
\begin{align*}
    \rho_{\varphi_{\vert \bfD\bfv\vert},\Omega}\big(\Dhk\bfz_h\big)&\leq c\,\rho_{\varphi,\Omega}\big(\Dhk\bfz_h\big)+c\,\rho_{\varphi,\Omega}(\bfD\bfv)\\&\leq c\,\big\|\Dhk\bfz_h\big\|_p^p+c\,\delta^p\vert \Omega\vert+c\,\rho_{\varphi,\Omega}(\bfD\bfv)
    \\&\leq c+c\,\delta^p\vert \Omega\vert+c\,\rho_{\varphi,\Omega}(\bfD\bfv)\\&\leq \max\big\{1,c+c\,\delta^p\vert \Omega\vert+c\,\rho_{\varphi,\Omega}(\bfD\bfv)\big\}\,,
\end{align*}
so that Lemma \ref{lem:just_to_mighty}, where $\lambda=1$ and $c_0=\max\{1,c+c\,\delta^p\vert \Omega\vert+c\,\rho_{\varphi,\Omega}(\bfD\bfv)\}$,~yields~that
\begin{align}\label{thm:error.7}
    I_{1,2}\leq \big(2\max\big\{1,c+c\,\delta^p\vert \Omega\vert+c\,\rho_{\varphi,\Omega}(\bfD\bfv)\big\}\big)^{\frac{1}{2}}\Delta_2(\varphi)\,.
\end{align}
Then, combining \eqref{thm:error.5} and \eqref{thm:error.7} in \eqref{thm:error.4}, we deduce that
\begin{align}\label{thm:error.I1}
    \vert I_1\vert 
  \leq c\,h 
  +c\,\big(\rho_{(\varphi_{\vert \bfD\bfv\vert})^*,\Omega}(h\,\nabla q)\big)^{\smash{\frac{1}{2}}}\,.
\end{align}

(\textit{ad $I_2$}).
Using the generalized Hölder inequality \eqref{eq:gen_hoelder}, we find that
\begin{align}\label{thm:error.8}
    \begin{aligned}
    \vert I_2\vert &\leq 2\,h^{\smash{\smash{1-\frac{1}{2}}}}\bignorm{\bfS_{\avg{\vert\Pia \Dhk\bfv_h\vert}}(h^{-1}\jump{\bfv_h\otimes\bfn})}_{(\varphi_{\smash{\avg{\vert\Pia \Dhk\bfv_h\vert}}})^*,\Gamma_h}\\&\quad\times\bignorm{h^{\smash{\smash{\frac{1}{2}-1}}}\jump{\bfz_h\otimes\bfn}}_{\varphi_{\smash{\avg{\vert\Pia \Dhk\bfv_h\vert}}},\Gamma_h}
    \\&
    \eqqcolon  2\,h^{\smash{\smash{1-\frac{1}{2}}}}\, I_{2,1}\cdot I_{2,2}\,.
    \end{aligned}
\end{align}
Appealing to \eqref{eq:hammera} and \cite[Theorem 4.1]{kr-pnse-ldg-2}, we have that 
\begin{align}\label{thm:error.9}
    \begin{aligned}
    &\rho_{(\varphi_{\smash{\avg{\vert\Pia
            \Dhk\bfv_h\vert}}})^*,\Gamma_h}\big(\bfS_{\avg{\vert\Pia
        \Dhk\bfv_h\vert}}(h^{-1}\jump{\bfv_h\otimes\bfn})\big)
    \\
    &\quad\leq c\,h^{-1}m_{\varphi_{\smash{\avg{\vert\Pia \Dhk\bfv_h\vert}}},h}(\bfv_h)
    \\
    &\quad     \leq c\,h\,\|\bfF(\bfD\bfv)\|_{1,2}^2+c\,h^{-1}\rho_{(\varphi_{\vert \bfD\bfv\vert})^*,\Omega}(h\,\nabla q)\,.
    \end{aligned}
\end{align}
Since, by assumption, we have that $h\leq 1$, for 
\begin{align*}
    c_0&\coloneqq \max\big\{1,c\,\|\bfF(\bfD\bfv)\|_{1,2}^2+c\,\rho_{(\varphi_{\vert \bfD\bfv\vert})^*,\Omega}(\nabla q)\big\}\ge 1\,,\\ \gamma&\coloneqq c_0^{-1}\big( c\,h\,\|\bfF(\bfD\bfv)\|_{1,2}^2+c\,h^{-1}\rho_{(\varphi_{\vert \bfD\bfv\vert})^*,\Omega}(h\,\nabla q)\big)\leq 1\,,
\end{align*}
Lemma \ref{lem:just_to_mighty} yields a constant $c(p)>0$, depending only on $p\in (2,\infty)$, such that
\begin{align}\label{thm:error.10}
    \begin{aligned}
    I_{2,1}\leq c_0^{\frac{1}{p'}-\frac{1}{2}}c(p)^{\frac{1}{p'}}\,\big(c\,h\,\|\bfF(\bfD\bfv)\|_{1,2}^2+c\,h^{-1}\rho_{(\varphi_{\vert \bfD\bfv\vert})^*,\Omega}(h\,\nabla q)\big)^{\frac{1}{2}}\,.
    \end{aligned}
\end{align}
Using that $\varphi_{\smash{\avg{\vert\Pia \Dhk\bfv_h\vert}}}(h^{\smash{\frac{1}{2}}}\,t)\leq c\,h\,\varphi_{\smash{\avg{\vert\Pia \Dhk\bfv_h\vert}}}(t)$ for all $t\ge 0$ and $h\in [0,1]$, the shift change in Lemma \ref{lem:shift-change}, that $\varphi(t)\leq c(p)\,(\delta^p+t^p)$~for~all~$t\ge 0$, $h\,\mathscr{H}^{d-1}(\Gamma_h)\leq c\,\vert \Omega\vert$, the discrete trace inequality \cite[(A.23)]{kr-phi-ldg}, the Orlicz-stability properties of $\Uppi_{h}^{0}$ \cite[(A.12)]{kr-phi-ldg}, and the a priori estimate \cite[Proposition~5.7]{kr-pnse-ldg-1}, we find that
\begin{align}\label{thm:error.12.0}
    \begin{aligned}
    \rho_{\varphi_{\smash{\avg{\vert\Pia
            \Dhk\bfv_h\vert}}},\Gamma_h}\big(h^{\smash{\frac{1}{2}-1}}\jump{\bfz_h\otimes\bfn}\big)
    &\leq c\, h\,\rho_{\varphi_{\smash{\avg{\vert\Pia
            \Dhk\bfv_h\vert}}},\Gamma_h}(h^{\smash{-1}}\jump{\bfz_h\otimes\bfn})
    \\
    &\leq
    c\,h\,\rho_{\varphi,\Gamma_h}(h^{-1}\jump{\bfz_h\otimes\bfn})
+c\,h\,\rho_{\varphi,\Gamma_h}\big(\avg{\vert\Pia \Dhk\bfv_h\vert}\big)\hspace*{-10mm}
    \\
    &\leq c\,h\,\|h^{-1}\jump{\bfz_h\otimes\bfn}\|_{p,\Gamma_h}^p+c\,\delta^p\,h\,\mathscr{H}^{d-1}(\Gamma_h)\\&\quad+
    c\,\rho_{\varphi,\Omega}\big(\Dhk\bfv_h\big)
    \\&\leq \max\big\{1,c+c\,\delta^p\,\vert\Omega\vert+c\,\smash{\|\bfg\|_{p'}^{p'}}\big\}\,,
    \end{aligned}
\end{align}
so that Lemma \ref{lem:just_to_mighty}, where  $\lambda=1$ and $c_0=\max\{1,c+c\,\delta^p\,\vert\Omega\vert+c\,\|\bfg\|_{p'}^{p'}\}$,~yields~that
\begin{align}\label{thm:error.12}
    I_{2,2}\leq \big(2\max\big\{1, c+c\,\delta^p\,\vert\Omega\vert + c\,\|\bfg\|_{p'}^{p'}\big\}\big)^{\smash{\frac{1}{2}}}\Delta_2(\varphi)\,.
\end{align}
Then, combining \eqref{thm:error.10} and \eqref{thm:error.12} in \eqref{thm:error.8}, we deduce that
\begin{align}\label{thm:error.13}
  \abs{I_2}\leq   c\,h 
  +c\,\big(\rho_{(\varphi_{\vert \bfD\bfv\vert})^*,\Omega}(h\,\nabla q)\big)^{\smash{\frac{1}{2}}}\,.
\end{align}

(\textit{ad $I_3$}). Introducing the notation $\bfe_h\coloneqq  \bfv_h-\bfv\in \WDG$, $I_3$ can be re-written~as
\begin{align}\label{eq:i3}
    \begin{aligned}
    I_3 &= b_h\big(\bfv,\bfv-\PiDG\bfv,\bfz_h\big)-b_h\big(\bfe_h,\PiDG\bfv,\bfz_h\big)+b_h\big(\bfv_h,\PiDG\bfe_h,\bfz_h\big)\\
    &\eqqcolon  I_{3,1}+I_{3,2}+I_{3,3}\,.
    \end{aligned}
\end{align}
So, we have to estimate $ I_{3,i}$, $i=1,2,3$:

(\textit{ad $I_{3,1}$}).  The definition of $b_h\colon\WDG\times \WDG\times\WDG\to \mathbb{R}$ 
yields: 
\begin{align}\label{eq:i31}
    \begin{aligned}
    2\,I_{3,1}=\big(\bfz_h\otimes \bfv,\Ghk(\bfv-\PiDG\bfv)\big)-\big((\bfv-\PiDG\bfv)\otimes \bfv,\Ghk\bfz_h\big)
    \eqqcolon  I_{3,1}^1+ I_{3,1}^2\,.
    \end{aligned}
\end{align}
As already observed in \cite[Proof of Lemma 2.6]{kr-pnse-ldg-2}, we have that
$\bfv\in W^{2,2}(\Omega)$, where 
\begin{align}\label{eq:w22}
 \norm{\bfv}_{2,2} \le c\,
\delta^{2-p}\, \norm{\bfF(\bfD\bfv)}_{1,2}\,.   
\end{align}
Thus, exploiting that, by the Sobolev
embedding theorem,  ${\bfv\in W^{2,2}(\Omega)\hookrightarrow
  L^\infty(\Omega)}$, the discrete Sobolev embedding theorem
(cf.~\cite[Theorem~5.3]{EP12}), the identities $\Ghk= \nabla _h -\Rhk$
and $\bfv-\PiDG\bfv=\bfv-\PiDG\bfv- \PiDG (\bfv-\PiDG\bfv)$, the
approximation properties of $\PiDG$ (cf. \cite[Corollary A.8, Lemma
A.1, Corollary A.19]{kr-phi-ldg}), and \eqref{eq:w22}, we find that
\begin{align}
    \label{eq:i311}
    \begin{aligned}
    \abs{I_{3,1}^1}&\leq \norm{\bfv}_\infty\norm{\bfz_h}_2\bignorm{\Ghk(\bfv-\PiDG\bfv)}_2\\
    &\leq c\,h\,\norm{\bfz_h}_{\nabla,p,h}\norm{\nabla^2\bfv}_2
    \\
    &\leq c\,h\,\norm{\nabla\bfF(\bfD\bfv)}_2\,.
    \end{aligned}
\end{align}
Similarly, we get, also using \eqref{eq:eqiv0} and $p>2$,  that
\begin{align}
    \label{eq:i312}
    \begin{aligned}
    \abs{I_{3,1}^2}&\leq \norm{\bfv}_\infty\bignorm{\bfv-\PiDG\bfv}_2\bignorm{\Ghk\bfz_h}_2\\
    &\leq c\,h^2\,\norm{\nabla^2\bfv}_2\norm{\bfz_h}_{\nabla,p,h}\\
    &\leq c\,\,h^2\,\norm{\nabla\bfF(\bfD\bfv)}_2\,.
    \end{aligned}
\end{align}

(\textit{ad $I_{3,2}$}). The definition of $b_h\colon\WDG\times \WDG\times\WDG\to \mathbb{R}$ 
yields:
\begin{align}
    \label{eq:i32}
    \begin{aligned}
    2\,I_{3,2}=\big(\bfz_h\otimes \bfe_h,\Ghk\PiDG\bfv\big)+\big(\PiDG\bfv\otimes \bfe_h,\Ghk\bfz_h\big)
    \eqqcolon  I_{3,2}^1+ I_{3,2}^2\,.
    \end{aligned}
\end{align}
Then, exploiting \eqref{eq:eqiv0}, the Sobolev embedding theorem
(cf.~Proposition \ref{prop:emb}), the DG-stability property of
$\PiDG$ (cf.~\cite[(A.19)]{dkrt-ldg}), the
Korn type inequality in Proposition~\ref{prop:kornii}, the estimates \eqref{eq:w22}, 
\cite[(4.50)]{kr-pnse-ldg-2}, and \cite[Theorem~4.1]{kr-pnse-ldg-2}, we find that
\begin{align}
    \label{eq:i321}
    \begin{aligned}
    \abs{I_{3,2}^1}&\leq
    \bignorm{\Ghk\PiDG\bfv}_2\norm{\bfe_h}_4\norm{\bfz_h}_4
    \\
    &\leq
    c\,\bignorm{\PiDG\bfv}_{\nabla,2,h}\norm{\bfe_h}_{\nabla,2,h}\norm{\bfz_h}_{\nabla,p,h}
    \\
    &\leq
    c\,\norm{\nabla\bfv}_{2}\big (\norm{\bfe_h}_{\bfD,2,h}+h\,\norm{\nabla ^2 \bfv}_2\big )\norm{\bfz_h}_{\nabla,p,h}
    \\
    &\leq
    c\,\big (\norm{\smash{\bfF(\bfD\bfv)
      -\bfF\big(\Dhk\bfv_h\big)}}_2+\big(m_{\phi_{\smash{\sssl}},h }
      (\bfv_h-\bfv)\big ) ^{\smash{\frac 12}}   +h\,\norm{\nabla ^2
        \bfv}_2\big )
      \hspace*{-10mm}
    \\[1mm]
    &\leq c\,\big(h\,\norm{\bfF(\bfD\bfv)}_{1,2}+\big(\rho_{(\varphi_{\vert \bfD\bfv\vert})^*,\Omega}(h\,\nabla q)\big)^{\smash{\frac{1}{2}}}\big)\,.
    \end{aligned}
\end{align}
Similarly, we find that
\begin{align}
    \label{eq:i322}
    \begin{aligned}
    \abs{I_{3,2}^2}&\leq \bignorm{\Ghk\bfz_h}_2\bignorm{\PiDG\bfv}_4\norm{\bfe_h}_4\\
    &\leq c\,\norm{\bfz_h}_{\nabla,p,h}\bignorm{\PiDG\bfv}_{\nabla,2,h}\norm{\bfe_h}_{\nabla,2,h}\\
    &\leq c\,\big(h\,\norm{\bfF(\bfD\bfv)}_{1,2}+\big(\rho_{(\varphi_{\vert \bfD\bfv\vert})^*,\Omega}(h\,\nabla q)\big)^{\smash{\frac{1}{2}}}\big)\,.
    \end{aligned}
\end{align}

(\textit{ad $I_{3,3}$}).  The definition of $b_h\colon\WDG\times \WDG\times\WDG\to \mathbb{R}$ 
yields:
\begin{align}
    \label{eq:i33}
    \begin{aligned}
    2\,I_{3,3}=\big(\bfz_h\otimes \bfv_h,\Ghk\PiDG\bfe_h\big)+\big(\PiDG\bfe_h\otimes \bfv_h,\Ghk\bfz_h\big)
    \eqqcolon  I_{3,3}^1+ I_{3,3}^2\,.
    \end{aligned}
\end{align}
Then, exploiting \eqref{eq:eqiv0}, the discrete Sobolev embedding
theorem (cf.~\cite[Theorem~5.3]{EP12}), the DG-stability properties of
$\PiDG$ (cf.~\cite[(A.18)]{dkrt-ldg}), the apriori estimate in
\cite[Proposition 5.7]{kr-pnse-ldg-1}, the
Korn type inequality in Proposition \ref{prop:kornii}, the estimates \eqref{eq:w22}, 
\cite[(4.50)]{kr-pnse-ldg-2}, and \cite[Theorem 4.1]{kr-pnse-ldg-2}, we find that
\begin{align}
    \label{eq:i331}
    \begin{aligned}
    \abs{I_{3,3}^1}&\leq \bignorm{\Ghk\PiDG\bfe_h}_2\norm{\bfz_h}_4\norm{\bfv_h}_4\\
    &\leq
    c\,\bignorm{\PiDG\bfe_h}_{\nabla,2,h}\norm{\bfz_h}_{\nabla,p,h}\norm{\bfv_h}_{\nabla,p,h}\\
    &\leq
    c\,\bignorm{\bfe_h}_{\nabla,2,h}\norm{\bfz_h}_{\nabla,p,h}\\
    &\leq c\,\big(h\,\norm{\bfF(\bfD\bfv)}_{1,2}+\big(\rho_{(\varphi_{\vert \bfD\bfv\vert})^*,\Omega}(h\,\nabla q)\big)^{\smash{\frac{1}{2}}}\big)\,.
    \end{aligned}
\end{align}
Similarly, we find that
\begin{align}
    \label{eq:i332}
    \begin{aligned}
    \abs{I_{3,3}^2}&\leq \bignorm{\Ghk\bfz_h}_2\bignorm{\PiDG\bfe_h}_4\norm{\bfv_h}_4\\
    &\leq c\,\norm{\bfz_h}_{\nabla,p,h}\norm{\bfe_h}_{\nabla,2,h}\norm{\bfv_h}_{\nabla,p,h}\\
    &\leq c\,\big(h\,\norm{\bfF(\bfD\bfv)}_{1,2}+\big(\rho_{(\varphi_{\vert \bfD\bfv\vert})^*,\Omega}(h\,\nabla q)\big)^{\smash{\frac{1}{2}}}\big)\,.
    \end{aligned}
\end{align}
Eventually, combining \eqref{eq:i3}--\eqref{eq:i332}, we conclude that
\begin{align}
  \abs{I_3}\leq c\,h 
  +c\,\big(\rho_{(\varphi_{\vert \bfD\bfv\vert})^*,\Omega}(h\,\nabla q)\big)^{\smash{\frac{1}{2}}}\,.\label{eq:i3fin}
\end{align}

(\textit{ad $I_4$}).
Using the generalized Hölder inequality \eqref{eq:gen_hoelder}, we find that
\begin{align}\label{thm:error.14}
    \begin{aligned}
    \vert I_4\vert &\leq 2\,h^{\smash{1-\frac{1}{2}}}\bignorm{\avg{\PiDG\bfS(\bfD\bfv)}-\avg{\bfS(\bfD\bfv)}}_{(\varphi_{\vert \bfD\bfv\vert})^*,\Gamma_h}\bignorm{h^{\smash{\frac{1}{2}-1}}\jump{\bfz_h\otimes\bfn}}_{\varphi_{\vert \bfD\bfv\vert},\Gamma_h}
    \\&\eqqcolon  2\, h^{\smash{1-\frac{1}{2}}}\, I_{4,1}\cdot I_{4,2}\,.
    \end{aligned}
\end{align}
Appealing to \cite[(4.43), (4.45)]{kr-pnse-ldg-2}, we have that 
\begin{align}\label{thm:error.15}
    \rho_{(\varphi_{\vert \bfD\bfv\vert})^*,\Gamma_h}\big(\avg{\PiDG\bfS(\bfD\bfv)}-\avg{\bfS(\bfD\bfv)}\big)&\leq c\,h^2\,\|\nabla \bfF(\bfD\bfv)\|_2^2\,.
\end{align}
Since, by assumption, we have that $h\leq 1$, for 
\begin{align*}
    c_0\coloneqq \max\big\{1,c\,\|\nabla\bfF(\bfD\bfv)\|_2^2\big\}\,,\quad \gamma\coloneqq c_0^{-1}c\,h^2\,\|\nabla\bfF(\bfD\bfv)\|_2^2\,,
\end{align*}
Lemma \ref{lem:just_to_mighty} yields a constant $c(p)>0$, depending only on $p\in (2,\infty)$, such that
\begin{align}\label{thm:error.16}
    I_{4,1}
  \leq c_0^{\frac{1}{p'}-\frac{1}{2}}\,c(p)\,h\,\|\nabla\bfF(\bfD\bfv)\|_2\,.
\end{align}
Using a shift change in Lemma \ref{lem:shift-change}, 
 \cite[Lemma 4.11]{kr-pnse-ldg-2}, \eqref{thm:error.12.0},
\cite[Theorem~4.1]{kr-pnse-ldg-2}, the convexity of $(\varphi_{\vert \bfD\bfv(x)\vert})^*\colon\mathbb{R}^{\ge 0}\hspace{-0.1em}\to\hspace{-0.1em} \mathbb{R}^{\ge 0}$ for a.e. $x\hspace{-0.1em}\in\hspace{-0.1em} \Omega$ to together with ${\sup_{a\ge 0}{\Delta_2((\varphi_a)^*)}\hspace{-0.1em}<\hspace{-0.1em}\infty}$ and $h\leq 1$, we find that
\begin{align}\label{thm:error.16.0}
    \begin{aligned}
    \rho_{\varphi_{\vert
        \bfD\bfv\vert},\Gamma_h}\big(h^{\smash{\frac{1}{2}-1}}\jump{\bfz_h\otimes\bfn}\big)
    &\leq  \rho_{\varphi_{\smash{\avg{\vert\Pia \Dhk\bfv_h\vert}}},\Gamma_h}\big(h^{\smash{\frac{1}{2}-1}}\jump{\bfz_h\otimes\bfn}\big)
    \\
    &\quad + \rho_{\varphi_{\vert\bfD\bfv\vert},\Gamma_h}\big(\vert\bfD\bfv\vert-
    \avg{\vert\Pia \Dhk\bfv_h\vert}\big)
    \\
    & \leq \rho_{\varphi_{\smash{\avg{\vert\Pia \Dhk\bfv_h\vert}}},\Gamma_h}\big(h^{\smash{\frac{1}{2}-1}}\jump{\bfz_h\otimes\bfn}\big)   
    \\
    &\quad +c\,h\,\norm{\nabla\bfF(\bfD\bfv)}^2_2+c\,h^{-1}\,\bignorm{\bfF(\Dhk\bfv_h)-\bfF(\bfD\bfv)}_2^2
    \\&\leq \max\big\{1,c+c\,\delta^p\,\vert \Omega\vert+c\,
    \norm{\bfg}_{p'}^{p'}+c\,h\,\|\bfF(\bfD\bfv)\|_{1,2}^2
    \\
    &\qquad\qquad\;+c\,\rho_{(\varphi_{\vert \bfD\bfv\vert})^*,\Omega}(\nabla q) \big\}\,.
    \end{aligned}
\end{align}
Thus,  Lemma \ref{lem:just_to_mighty}, where  $\gamma=1$ and $c_0=\max\{1,c+c\delta^p\,\vert \Omega\vert+c\, \norm{\bfg}_{p'}^{p'}+c\,\|\bfF(\bfD\bfv)\|_{1,2}^2+c\,\rho_{(\varphi_{\vert \bfD\bfv\vert})^*,\Omega}(\nabla q)\big\}$,~yields~that
\begin{align}\label{thm:error.17}
    \hspace*{-1mm}I_{4,2}
  &\!\leq \! \big(2\max \!\big\{1,c\!+\!c\delta^p\,\vert
    \Omega\vert\!+\!c \norm{\bfg}_{p'}^{p'}
    \!+\!c\|\bfF(\bfD\bfv)\|_{1,2}^2\!+\!c\rho_{(\varphi_{\vert
    \bfD\bfv\vert})^*,\Omega}(\nabla
    q)\big\}\big)^{\smash{\frac{1}{2}}} \Delta_2(\varphi). \hspace*{-2mm}
\end{align}
Combining \eqref{thm:error.16} and \eqref{thm:error.17}, we deduce that
\begin{align}\label{thm:error.18}
    \abs{I_4}\leq c\,h\,. 
\end{align}               

(\textit{ad $I_5$}).
Using  Hölder's inequality, we find that
\begin{align}\label{thm:error.19}
  \vert I_5\vert
  &\leq  h^{\smash{1-\frac{1}{2}}}\bignorm{\avg{\bfK}-\avg{\PiDG
    \bfK}}_{2,\Gamma_h}\bignorm{h^{\smash{\frac{1}{2}-1}}\jump{\bfz_h\otimes\bfn}}_{2,\Gamma_h}\,.
\end{align}
    Taking into account $\bfv\in L^\infty(\Omega)\cap W^{1,2}_0(\Omega)$, we
    have that $\bfK =\bfv \otimes \bfv\in W^{1,2}_0(\Omega)$, where
    $\norm{\nabla     \bfK}_{2} \le c\,
    \norm{\bfv}_{\infty}\,\delta^{2-p}\|\bfF(\bfD\bfv)\|_{2} $. Thus,
    \cite[Corollary A.19]{kr-phi-ldg} yields
\begin{align}\label{thm:error.20}
    \bignorm{\avg{\bfK}-\avg{\PiDG \bfK}}_{2,\Gamma_h}\leq
  c\,h^{\frac 12}\,\norm{\nabla \bfK}_{2}\le c\,h^{\frac 12}\,\norm{\bfF(\bfD\bfv)}_{2}\,,
\end{align}
which together with
$\bignorm{h^{\smash{\frac{1}{2}-1}}\jump{\bfz_h\otimes\bfn}}_{2,\Gamma_h}
\le c\,\abs{\Omega}^{\frac 12 -\frac 1p}
\bignorm{h^{\smash{\frac{1}{p}-1}}\jump{\bfz_h\otimes\bfn}}_{p,\Gamma_h}\le
c$ yields
\begin{align}
    \vert I_5\vert
    \leq c\,h\,. 
\end{align}

(\textit{ad $I_6$}).
Using the generalized Hölder inequality \eqref{eq:gen_hoelder}, we find that
\begin{align}
    \vert I_6\vert &\leq 2\,h^{\smash{1-\frac{1}{2}}}\bignorm{\avg{q\mathbf{I}_d}-\avg{\PiDG(q\mathbf{I}_d)}}_{(\varphi_{\vert \bfD\bfv\vert})^*,\Gamma_h}\bignorm{h^{\smash{\frac{1}{2}-1}}\jump{\bfz_h\otimes\bfn}}_{\varphi_{\vert \bfD\bfv\vert},\Gamma_h}
    \\&\eqqcolon  2\, h^{\smash{1-\frac{1}{2}}}\, I_{6,1}\cdot I_{4,2}\,.
\end{align}
Appealing to \cite[Proposition 4.9]{kr-pnse-ldg-2}, we have that 
\begin{align*}
    \rho_{(\varphi_{\vert
  \bfD\bfv\vert})^*,\Gamma_h}\big(\avg{q\mathbf{I}_d}-\avg{\PiDG(q\mathbf{I}_d)}\big)
  &\leq c\,h\,\|\nabla\bfF(\bfD\bfv)\|_{2}^2+c\,h^{-1}\,\rho_{(\varphi_{\vert \bfD\bfv\vert})^*,\Omega}(h\,\nabla q)\,.
\end{align*}
Since, by assumption, we have that $h\leq 1$, for 
\begin{align*}
    c_0&\coloneqq \max\{1,c\,\|\bfF(\bfD\bfv)\|_{1,2}^2+c\,\rho_{(\varphi_{\vert \bfD\bfv\vert})^*,\Omega}(\nabla q)\}\,,\\ \gamma&\coloneqq c_0^{-1}\big(c\,h\|\bfF(\bfD\bfv)\|_{1,2}^2+c\,h^{-1}\rho_{(\varphi_{\vert \bfD\bfv\vert})^*,\Omega}(h\,\nabla q)\big)\,,
\end{align*}
Lemma \ref{lem:just_to_mighty} yields a constant $c(p)>0$, depending only on $p\in (1,\infty)$,~such~that
\begin{align*}
    &\big\|\avg{q\mathbf{I}_d}\!-\!\avg{\PiDG(q\mathbf{I}_d)}\big\|_{(\varphi_{\vert
      \bfD\bfv\vert})^*,\Gamma_h}
\leq c_0^{\frac{1}{p'}-\frac{1}{2}} c\,\big(c\,h\,\|\nabla\bfF(\bfD\bfv)\|_{2}^2+c\,h^{-1}\rho_{(\varphi_{\vert \bfD\bfv\vert})^*,\Omega}(h\,\nabla q)\big)^{\frac{1}{2}}\,.
\end{align*}
As a result, also using \eqref{thm:error.17},
we deduce that
\begin{align}\label{thm:error.I6}
  \abs{I_6}\leq c\,h 
  +c\,\big (\rho_{(\varphi_{\vert \bfD\bfv\vert})^*,\Omega}(h\,\nabla q)\big)^{\frac{1}{2}}\,.
\end{align}

Putting it all together, for every $\bfz_h\in \Vhk$ with $\|\bfz_h\|_{\nabla,p,h}\leq 1$, we conclude that
\begin{align}
  \hskp{q_h-q}{\Divhk \bfz_h}\leq  c\,h 
  +c\,\big(\rho_{(\varphi_{\vert \bfD\bfv\vert})^*,\Omega}(h\,\nabla q)\big)^{\smash{\frac{1}{2}}}\,.
\end{align}
Therefore, for every ${z_h\in \Qhkco}$, we find that
\begin{align}\label{eq:fast}
  \begin{aligned}
    \|q_h-q\|_{p'}&\leq \|q_h-z_h\|_{p'}+\|z_h-q\|_{p'} \\&\leq
    \sup_{\bfz_h\in \Vhk; {\norm{\bfz_h}_{\nabla,p,h}\leq
        1}}{(q_h-z_h,\Divhk\bfz_h)}+\|z_h-q\|_{p'} \\&\leq
    \sup_{\bfz_h\in \Vhk; {\norm{\bfz_h}_{\nabla,p,h}\leq
        1}}{(q_h-q,\Divhk\bfz_h)}+c\,\|z_h-q\|_{p'} \\&\leq
    c\,h 
    +c\,\big(\rho_{(\varphi_{\vert \bfD\bfv\vert})^*,\Omega}(h\,\nabla
    q)\big)^{\smash{\frac{1}{2}}}+c\,\|z_h-q\|_{p'}\,.
  \end{aligned}
\end{align}
Next,  denote by $\Pi_h^{\smash{Q},k}\colon L^{p'}(\Omega)\hspace{-0.1em}\to\hspace{-0.1em} \Qhkc$, the
Clem\'ent quasi-interpolation~operator~(cf.~\cite{BF1991}), for which we have that
\begin{align}\label{thm:error.5a}
  \|q-\Pi_h^{\smash{Q},k} q\|_{p'}\leq
  c\,h \,\norm{\nabla q}_{p'}\,.
\end{align}
Noting that the infimum of $\|z_h-q\|_{p'}$ over $\Qhkc$ and $\Qhkco$
are comparable for $q \in \Qo$, the assertion of Theorem
\ref{thm:error} follows from \eqref{eq:fast} and \eqref{thm:error.5a}, if we choose $\bfz_h=\Pi_h^{\smash{Q},k} q$.
\end{proof}

\begin{proof}[Proof of Corollary \ref{cor:error}]
    Using that $\varphi^*(h\,t)\hspace{-0.1em} \le\hspace{-0.1em}
    c\, h^{p'} \varphi^*(t)$ for all $t\hspace{-0.1em}\ge\hspace{-0.1em} 0$, valid~for~${p\hspace{-0.1em}>\hspace{-0.1em}2}$ (cf.~\cite{bdr-phi-stokes}),  we deduce from Theorem \ref{thm:error} that
    \begin{align*}
      \|q_h-q\|_{p'}\le c\,h 
      +c\,h^{\frac{p'}2} \big (\rho_{\phi^*,\Omega}(\nabla q)\big
      )^{\frac 12}\,. 
    \end{align*}
    If, in addition, $\bfg\in L^2(\Omega)$, then \cite[Lemma
    2.6]{kr-pnse-ldg-2} implies $(\delta
    +\abs{\bfD\bfv})^{2-p}\abs{\nabla q}^2 \in  L^1(\Omega)$. Moreover, it holds
  $(\varphi_a)^*(h\,t)\hspace{-0.15em} \sim \hspace{-0.15em}\big( (\delta+a)^{p-1}\hspace{-0.15em}
  +h\,t\big )^{\smash{p'-2}}\, h^2\, t^2 \hspace{-0.15em}
  \le\hspace{-0.15em}(\delta+a)^{2-p}\, h^2\, t^2 $ for all
  $t,a\hspace{-0.15em}\ge\hspace{-0.15em} 0$, since $p>2$. Therefore, 
   from
  Theorem \ref{thm:error},    we deduce that
    \begin{align*}
       \|q_h-q\|_{p'}\le c\, h 
      +c\, h\, \norm {(\delta+\abs{\bfD\bfv})^{\frac{2-p}2}\nabla
      q}_2\,, 
    \end{align*}
    which is the assertion.
\end{proof}

The same method of proof of course also works for the $p$-Stokes
problem, i.e., we neglect the convective term in
\eqref{eq:p-navier-stokes}, Problem (P), Problem (P$_h$), Problem (Q),
and Problem (Q$_h$). Note that the dependence on $\delta^{-1}>0$ comes
solely from the convective term. Thus, we obtain for the $p$-Stokes problem a better dependence on the constants. 
\begin{theorem}
  \label{thm:error_stokes}
  Let $\SSS$ satisfy Assumption~\ref{assum:extra_stress} with
  $p\in(2,\infty)$ and $\delta\ge 0$,~let~$k\in \mathbb{N}$, and let
  $\bfg\in L^{p'}(\Omega)$. Moreover, let
  $(\bfv,q)^\top \in \Vo(0)\times \Qo$ be a solution of
  Problem~(Q) without the convective term (cf.~\eqref{eq:q1}, \eqref{eq:q2}) with
  $\bfF(\bfD \bfv) \in W^{1,2}(\Omega)$ 
  and let $(\bfv_h,q_h)^\top \in \Vhk(0)\times \Qhkco$ be a solution
  of Problem (Q$_h$) without the terms coming from the convective term (cf.~\eqref{eq:primal1}) for $\alpha>0$. 
  Then, there exists a constant $c>0$, depending only on the characteristics of
  $\SSS$, $\norm{\bfF(\bfD\bfv)}_{1,2}$, $\norm{\nabla q}_{p'}$,
  $\norm{\bfg}_{p'}$, 
  $\delta^p\abs{\Omega}$,~${\omega_0}$,~${\alpha^{-1}}$,~and~$k$, such that
  \begin{align*}
    \smash{\|q_h-q\|_{p'}\le c\,h 
    +c\, \big(\rho_{(\phi_{\vert\bfD\bfv\vert})^*,\Omega}(h\,\nabla q)\big)^{\frac{1}{2}}}\,.
  \end{align*}
\end{theorem}

\begin{corollary}\label{cor:error_stokes}
  Let the assumptions of Theorem \ref{thm:error_stokes} be satisfied. Then,~it~holds
  \begin{align}\label{eq:p'}
    \|q_h-q\|_{p'}\le c\, h^{\frac{p'}{2}} 
  \end{align}
  with a constant $c>0$ depending only on the characteristics of
  $\SSS$, $\norm{\bfF(\bfD\bfv)}_{1,2}$, $\norm{\nabla q}_{p'}$,
  $\norm{\bfg}_{p'}$, 
  $\delta^p\abs{\Omega}$,  $\omega_0$,
  $\alpha^{-1}$, and $k$.  If, in
  addition,~${\bfg\in L^2(\Omega)}$, then
  \begin{align}\label{eq:p'a}
    \|q_h-q\|_{p'}\leq c\, h
  \end{align}
  with a constant $c>0$ depending only on the characteristics of
$\SSS$, $\norm{\bfF(\bfD\bfv)}_{1,2}$, 
  $\norm{\bfg}_{p'}$,  $\norm{(\delta+\vert
    \bfD\bfv\vert)^{\smash{\frac{2-p}{2}}}\nabla q}_2$, 
  $\delta^p\abs{\Omega}$, 
  $\omega_0$, 
  $\alpha^{-1}$, and $k$.
\end{corollary}

  \begin{remark}\label{rem:pse}
    {\rm
      It seems that the only other result for DG methods proving for the $p$-Stokes
      problem ($p \neq 2$) an error estimate
      for the pressure is \cite{BCPH20}. There it is proved that $ \|q_h-q\|_{p'}\leq c\, h^{\frac {k+1}{p-1}} $ if
      $\bfv \in W^{k+2,p}(\Omega)$, $q  \in W^{k+1,p'}(\Omega)$,
      $k\in \setN$, which requires at least one order higher regularity
      compared to our results. In the special case that the extra
      stress tensor has $(2,\delta)$-structure, i.e., it is non-linear
      and possesses linear growth, it is proved in
      \cite{bustinza-fluid,CHSW13,GS15} that the pressure has linear
      order of convergence, which agrees with our results. The results in Corollary \ref{cor:error}
      can also be found in the context of FE methods. In particular,  estimate
      \eqref{eq:p'} is proved in \cite{BL1994b,bdr-phi-stokes,Hi13a}, while 
      estimate \eqref{eq:p'a} is proved in \cite{San1993}. Even for FE
    methods there are no theoretical results proving the experimentally
    observed convergence rates in the case $p>2$ (cf.~Section~\ref{sec:experiments}).
}
  \end{remark}

\begin{proof}[Proof of Theorem \ref{thm:error_stokes}]
  We proceed analogously to the proof of Theorem \ref{thm:error}. In
  view of the absense of the convective term, the equality
  \eqref{thm:error.2} now reads
  \begin{align}
    \label{eq:2}
  \begin{aligned}
    \hskp{q_h-q}{\Divhk \bfz_h}
    &=I_1+\alpha\, I_2+I_4+I_6\,,
  \end{aligned}
  \end{align}
  where $I_i$, $i=1,2,4,6$ are defined in \eqref{thm:error.2}. Then,
  resorting~in~\eqref{eq:2} to \eqref{thm:error.I1},
  \eqref{thm:error.13}, \eqref{thm:error.18}, \eqref{thm:error.I6}, we
  conclude that for every $\bfz_h\in \Vhk$ with
  $\|\bfz_h\|_{\nabla,p,h}\leq 1$, we have that
\begin{align}\label{eq:fast1}
  \hskp{q_h-q}{\Divhk \bfz_h}\leq  c\,h 
  +c\,\big(\rho_{(\varphi_{\vert \bfD\bfv\vert})^*,\Omega}(h\,\nabla q)\big)^{\smash{\frac{1}{2}}}
\end{align}
  with a constant $c\hspace{-0.1em}>\hspace{-0.1em}0$ depending only on the characteristics of
 $\SSS$, $\norm{\bfF(\bfD\bfv)}_{1,2}$, $\norm{\nabla q}_{p'}$,
  $\norm{\bfg}_{p'}$, 
  $\delta^p\abs{\Omega}$, $\omega_0$, 
  $\alpha^{-1}$, and $k$. Having at our disposal
  \eqref{eq:fast1}, we conclude the proof as in the proof of Theorem \ref{thm:error}.
\end{proof}

\begin{proof}[Proof of Corollary \ref{cor:error_stokes}]
    We follow the arguments in the proof of Corollary \ref{cor:error}, but now resort to Theorem \ref{thm:error_stokes}.
\end{proof}

\section{Numerical experiments}\label{sec:experiments}

In this section, we apply the LDG scheme \eqref{eq:DG} (or \eqref{eq:primal1} and  \eqref{eq:primal2}) to solve
numerically the system~\eqref{eq:p-navier-stokes}~with $\SSS\colon\mathbb{R}^{d\times d}\to\mathbb{R}^{d\times d}$, for every $\bfA\in\mathbb{R}^{d\times d}$ defined via  ${\SSS(\bfA) \coloneqq (\delta+\vert \bfA^{\textup{sym}}\vert)^{p-2}\bfA^{\textup{sym}}}$,
where $\delta\coloneqq 1\textrm{e}{-}4$ and $ p>2$.
We approximate the discrete solution~${\bfv_h\in V^k_h}$ of the
non-linear problem~\eqref{eq:DG}~by~deploying the Newton solver from \mbox{\textsf{PETSc}} (version 3.17.3), cf.~\cite{LW10}, with an absolute tolerance~of $\tau_{abs}\!=\! 1\textrm{e}{-}8$ and a relative tolerance of $\tau_{rel}\!=\!1\textrm{e}{-}10$. The linear system emerging in each Newton step is solved using a sparse direct solver from \textsf{MUMPS} (version~5.5.0),~cf.~\cite{mumps}. For the numerical flux \eqref{def:flux-S}, we choose the fixed parameter $\alpha=2.5$. This choice is in accordance with the choice in \mbox{\cite[Table~1]{dkrt-ldg}}. In the implementation, the uniqueness of the pressure is enforced via a zero mean condition.

All experiments were carried out using the finite element software package~\mbox{\textsf{FEniCS}} (version 2019.1.0), cf.~\cite{LW10}. 

For \hspace{-0.1mm}our \hspace{-0.1mm}numerical \hspace{-0.1mm}experiments, \hspace{-0.1mm}we \hspace{-0.1mm}choose \hspace{-0.1mm}$\Omega\!=\! (-1,1)^2$ \hspace{-0.1mm}and \hspace{-0.1mm}linear~\hspace{-0.1mm}elements,~\hspace{-0.1mm}i.e.,~\hspace{-0.1mm}${k\!=\! 1}$. We choose $\smash{\bfg\in L^{p'}(\Omega)}$ and boundary data $\bfv_0\in W^{\smash{1,1-\frac{1}{p}}}(\partial\Omega)$\footnote{The exact solution is not zero on the boundary of the computational domain. However, the error is clearly concentrated around the singularity and, thus, this small inconsistency with the setup of the theory does not have any influence on the results of this paper. In particular, note that Part I of the paper (cf.~\cite{kr-pnse-ldg-1})
already established at least the weak convergence of the method also for the fully non-homogeneous case.\vspace{-7.5mm}}  such that $\bfv\in W^{1,p}(\Omega)$ and $q \in \Qo$, for every $x\coloneqq (x_1,x_2)^\top\in \Omega$ defined by
\begin{align}
	\bfv(x)\coloneqq \vert x\vert^{\beta} (x_2,-x_1)^\top\,, \qquad q(x)\coloneqq \eta\,(\vert x\vert^{\gamma}-\langle\,\vert \!\cdot\!\vert^{\gamma}\,\rangle_\Omega)
\end{align}
are a solutions of  \eqref{eq:p-navier-stokes}. Here, we choose $\beta=1\textrm{e}{-}2$, which implies ${\bfF(\bfD\bfv)\in W^{1,2}(\Omega)}$. Concerning the pressure regularity, we consider two cases: Namely, 
we choose either $\gamma= 1-{\frac{2}{p'}}+1\textrm{e}{-}4$ and $\eta=25$,~which just yields ${q \in
  W^{1,p'}(\Omega)}$ (case 1),~or we choose $\gamma= \alpha\frac{p-2}{2}+1\textrm{e}{-}4$ and $\eta=1\textrm{e}{+}3$, which just yields $(\delta+\abs{\bfD\bfv})^{\frac{2-p}{2}}\nabla q \in
  L^2(\Omega)$~(case~2). Thus,~for $\gamma= 1-{\frac{2}{p'}}+1\textrm{e}{-}4$ and $\eta=25$ (case 1), we can expect the
convergence rate $\smash{\frac{p'}{2}}$, while for $\gamma\hspace{-0.1em}= \hspace{-0.1em}\alpha\frac{p-2}{2}+1\textrm{e}{-}4$ and $\eta\hspace{-0.1em}=\hspace{-0.1em}1\textrm{e}{+}3$ (case 2), we can expect the convergence~rate~$1$. (cf.~Corollary~\ref{cor:error}).%
 
We \hspace{-0.1mm}construct \hspace{-0.1mm}a \hspace{-0.1mm}initial \hspace{-0.1mm}triangulation \hspace{-0.1mm}$\mathcal
T_{h_0}$, \hspace{-0.1mm}where \hspace{-0.1mm}$h_0\hspace{-0.2em}=\hspace{-0.2em}\smash{\frac{1}{\sqrt{2}}}$, \hspace{-0.1mm}by \hspace{-0.1mm}subdividing~\hspace{-0.1mm}a~\hspace{-0.1mm}\mbox{rectangular} cartesian grid~into regular triangles with different orientations.  Finer triangulations~$\mathcal T_{h_i}$, $i=1,\dots,5$, where $h_{i+1}=\frac{h_i}{2}$ for all $i=1,\dots,5$, are 
obtained by
regular subdivision of the previous grid: Each \mbox{triangle} is subdivided
into four equal triangles by connecting the midpoints of the edges, i.e., applying the red-refinement rule, cf. \cite[Definition~4.8~(i)]{Ba16}.

Then, for the resulting series of triangulations $\mathcal T_{h_i}$, $i\!=\!1,\dots,5$, we apply~the~above Newton scheme to compute the corresponding numerical solutions $(\bfv_i,\bfL_i,\bfS_i)^\top\coloneqq \smash{(\bfv_{h_i},\bfL_{h_i},\bfS_{h_i})^\top\in V_{h_i}^k\times X_{h_i}^k\times X_{h_i}^k}$, $i=1,\dots,5$, 
and the error quantities
\begin{align*}
		e_{q,i}&\coloneqq \|q_i-q\|_{p'}\,,
	\quad i=1,\dots,5\,.
\end{align*}
As estimation of the convergence rates,  the experimental order of convergence~(EOC)
\begin{align*}
	\texttt{EOC}_i(e_{q,i})\coloneqq \frac{\log(e_{q,i}/e_{q,i-1})}{\log(h_i/h_{i-1})}\,, \quad i=1,\dots,5\,,
\end{align*}
is recorded.  
For different values of $p\in \{2.25, 2.5, 2.75, 3, 3.25, 3.5\}$ and a
series of triangulations~$\mathcal{T}_{h_i}$, $i = 1,\dots,5$,
obtained by regular, global refinement as described above, the EOC is
computed and presented
in 
Table~\ref{tab4}. In it, we observe for  case 1 a
  convergence~ratio~of~about $\texttt{EOC}_i(e_{q,i})\approx 1$,
  $i=1,\dots, 5$, and in  case 2  a convergence~ratio~of~about
$\texttt{EOC}_i(e_{q,i})\approx \smash{\frac{2}{p'}}$, $i=1,\dots,
5$. Both are higher than the proved~convergence rates \eqref{eq:p'1},
\eqref{eq:p'a1} in Corollary~\ref{cor:error}. The same
  convergence rate as in case 1 for the $p$-Stokes problem is observed in the analogous numerical
  experiment in \cite{bdr-phi-stokes} in the context of FE methods. 
This indicates that the error estimates in Corollary~\ref{cor:error}
and Corollary~\ref{cor:error_stokes} are yet sub-optimal and it
might, therefore, be possible to improve the estimates
  \eqref{eq:p'1}, \eqref{eq:p'} to $\smash{\|q_i-q\|_{p'}\leq c\,h}$
  and the estimates \eqref{eq:p'a1}, \eqref{eq:p'a} to
  $\smash{\|q_i-q\|_{p'}\leq c\,h^{2/p'}}$. However, without
  additional regularity assumptions, such a result seems to be out of
  reach at the present time due to the disbalance of the shifts in the
  various terms.

\begin{table}[H]
    \setlength\tabcolsep{1.9pt}
	\centering
	\begin{tabular}{c |c|c|c|c|c|c|c|c|c|c|c|c|} \cmidrule(){1-13}
\multicolumn{1}{|c||}{\cellcolor{lightgray}$\gamma$}	
	& \multicolumn{6}{c||}{\cellcolor{lightgray}case 1}   & \multicolumn{6}{c|}{\cellcolor{lightgray}case 2}\\ 
			\hline 
		   
		    \multicolumn{1}{|c||}{\cellcolor{lightgray}\diagbox[height=1.1\line,width=0.11\dimexpr\linewidth]{\vspace{-0.6mm}$i$}{\\[-5mm] $p$}}
		    & \cellcolor{lightgray}2.25 & \cellcolor{lightgray}2.5  & \cellcolor{lightgray}2.75  &  \cellcolor{lightgray}3.0 & \cellcolor{lightgray}3.25  & \multicolumn{1}{c||}{\cellcolor{lightgray}3.5} &  \multicolumn{1}{c|}{\cellcolor{lightgray}2.25}   & \cellcolor{lightgray}2.5  & \cellcolor{lightgray}2.75  & \cellcolor{lightgray}3.0  & \cellcolor{lightgray}3.25 &   \cellcolor{lightgray}3.5 \\ \hline\hline
			\multicolumn{1}{|c||}{\cellcolor{lightgray}$1$}                		& 0.988 & 0.986 & 0.984 & 0.983 & 0.983 & \multicolumn{1}{c||}{0.982} & \multicolumn{1}{c|}{1.096} & 1.175 & 1.237 & 1.285 & 1.324 & 1.357 \\ \hline
			\multicolumn{1}{|c||}{\cellcolor{lightgray}$2$}                  	& 0.997 & 0.995 & 0.994 & 0.993 & 0.992 & \multicolumn{1}{c||}{0.991} & \multicolumn{1}{c|}{1.107} & 1.191 & 1.258 & 1.312 & 1.356 & 1.392 \\ \hline
			\multicolumn{1}{|c||}{\cellcolor{lightgray}$3$}                     & 0.999 & 0.999 & 0.998 & 0.997 & 0.997 & \multicolumn{1}{c||}{0.996} & \multicolumn{1}{c|}{1.111} & 1.198 & 1.267 & 1.323 & 1.370 & 1.410 \\ \hline
			\multicolumn{1}{|c||}{\cellcolor{lightgray}$4$}               		& 1.000 & 1.000 & 0.999 & 0.999 & 0.999 & \multicolumn{1}{c||}{0.998} & \multicolumn{1}{c|}{1.112} & 1.201 & 1.272 & 1.322 & 1.364 & 1.403 \\ \hline
			\multicolumn{1}{|c||}{\cellcolor{lightgray}$5$}               		& 1.000 & 1.000 & 1.000 & 0.999 & 0.998 & \multicolumn{1}{c||}{0.998} & \multicolumn{1}{c|}{1.112} & 1.202 & 1.277 & 1.324 & 1.323 & 1.334 \\ \hline\hline
			\multicolumn{1}{|c||}{\cellcolor{lightgray}\small expected}         & 0.900 & 0.833 & 0.786 & 0.750 & 0.722 & \multicolumn{1}{c||}{0.700} & \multicolumn{1}{c|}{1.000} & 1.000 & 1.000 & 1.000 & 1.000 & 1.000 \\ \hline
	\end{tabular}\vspace{-2mm}
	\caption{Experimental order of convergence: $\texttt{EOC}_i(e_{q,i})$,~${i=1,\dots,5}$.}
	\label{tab4}
\end{table}\vspace{-1cm}

\end{document}

%% file: part3_shared.tex
\usepackage{lipsum}
\usepackage{amsfonts}
\usepackage{graphicx}
\usepackage{epstopdf}
\usepackage{algorithmic}
\ifpdf
  \DeclareGraphicsExtensions{.eps,.pdf,.png,.jpg}
\else
  \DeclareGraphicsExtensions{.eps}
\fi

\newsiamremark{remark}{Remark}
\newsiamremark{hypothesis}{Hypothesis}
\crefname{hypothesis}{Hypothesis}{Hypotheses}
\newsiamthm{claim}{Claim}

\headers{LDG approximation for the \lowercase{$p$}-Navier--Stokes system}{A. Kaltenbach, M. R\r{U}\v{Z}I\v{C}KA}

\title{A Local Discontinuous Galerkin approximation for the \lowercase{$p$}-Navier--Stokes system, Part III: Convergence rates for the pressure\thanks{Submitted to the editors \today.
}}

\author{Alex Kaltenbach\thanks{Department of Applied Mathematics, Albert--Ludwigs--University Freiburg, 79104 (Germany)
  (\email{alex.kaltenbach@mathematik.uni-freiburg.de}).}
\and Michael R\r{U}\v{Z}I\v{C}KA\thanks{Department of Applied Mathematics, Albert--Ludwigs--University Freiburg, 79104 (Germany)
  (\email{rose@mathematik.uni-freiburg.de}).}}

\usepackage{amsopn}
